\documentclass{amsart}
\usepackage[utf8]{inputenc}		
\usepackage{amssymb}
\usepackage{lastpage}			
\usepackage{color,xcolor}		
\usepackage{graphicx}			
\usepackage{ytableau}
\usepackage{pgf,tikz}
\usepackage{caption}
\usepackage{placeins}
\usepackage{enumitem}
\usepackage{amsmath}
\usepackage{amsthm}			
\usepackage{hyperref}			
\usepackage{cleveref}			
\usepackage{mathabx}	
\newtheorem{theo}{Theorem} 
\newtheorem{lem}[theo]{Lemma}
\newtheorem{prop}[theo]{Proposition} 
\newtheorem{cor}[theo]{Corollary} 

\newtheorem{defi}{Definition} 
\newtheorem{notation}{Notation}
 
\newtheorem{exam}{Example}
\DeclareMathOperator{\Gdeg}{\mathrm{G-deg}}
\DeclareMathOperator{\ch}{\mathrm{char}}
\DeclareMathOperator{\spn}{\mathrm{span}}

\begin{document}
\title[Specht property of graded Lie algebras]{Specht property of varieties of graded Lie algebras}
\author{Daniela Martinez Correa}\address{Department of Mathematics, State University of Campinas, 651 S\'ergio Buarque de Holanda, 13083-859 Campinas, SP, Brazil}\email{d190688@dac.unicamp.br}

\author{Plamen Koshlukov}\address{Department of Mathematics, State University of Campinas, 651 S\'ergio Buarque de Holanda, 13083-859 Campinas, SP, Brazil}\email{plamen@unicamp.br} 
\thanks{D. Correa was supported by PhD grant from CAPES, Brazil. P. Koshlukov was partially supported by grant No. 2018/23690-6 from FAPESP, Brazil, and by  grant No. 302238/2019-0 from CNPq, Brazil.}

\begin{abstract}
Let $UT_n(F)$ be the algebra of the $n\times n$ upper triangular matrices and denote $UT_n(F)^{(-)}$  the Lie algebra on the vector space of $UT_n(F)$ with respect to the usual bracket (commutator), over an infinite field $F$.
In this paper, we give a positive answer to the Specht property for the ideal of the $\mathbb{Z}_n$-graded identities of $UT_n(F)^{(-)}$ with the canonical grading when the characteristic $p$ of $F$ is 0 or is larger than $n-1$. Namely we prove that every ideal of graded identities in the free graded Lie algebra that contains the graded identities of $UT_n(F)^{(-)}$, is finitely based. 

Moreover we show that if $F$ is an infinite field of characteristic $p=2$ then the $\mathbb{Z}_3$-graded identities of $UT_3^{(-)}(F)$ do not satisfy the Specht property. More precisely, we construct explicitly an ideal of graded identities containing that of $UT_3^{(-)}(F)$, and which is not finitely generated as an ideal of graded identities.

\end{abstract}\maketitle
\section{Introduction}
One of the most important problems in the theory of algebras with polynomial identities is that of determining the identities of  specific algebras and studying the properties of the varieties that these algebras generate. The most significant part of the advances in this area has been obtained for associative algebras over fields of characteristic zero. 
Although the study of problems in positive characteristic and for other types of algebras has grown in the last decades, there are  still very many questions to answer.

In 1984--1986, Kemer proved in \cite{Ke} that for every associative algebra over a field of characteristic zero, its T-ideal is finitely generated as a T-ideal, thus providing a positive answer to the long standing Specht problem. To that end Kemer developed a sophisticated theory which described the structure of the ideals of identities in the free associative algebra. We observe that Kemer's theory has been shaping a good deal of the research in PI theory since then. 

Here we recall that for a wide range of groups and algebras the analogue of the Specht problem was investigated and solved. We mention the paper by Oates and Powell \cite{oates_powell}, who proved that the variety of groups generated by a finite group admits a finite basis of its laws. It is also known that the variety generated by a finite ring satisfies the Specht property; this was obtained independently by Kruse \cite{kruse} and by Lvov \cite{lvov}. Bahturin and Ol'shanski\u{\i} proved in \cite{baht_ol} that the identities of a finite Lie ring or Lie algebra also admit a finite basis of its identities.

On the other hand, if the base field is infinite but of positive characteristic, the  Specht problem can have a negative answer. The first such examples for Lie algebras were obtained by Vaughan-Lee \cite{VL}, in characteristic 2, and by Drensky \cite{dr_inf}, for every characteristic $p>0$.  The first examples in the case of associative algebras were obtained, much later, independently (and almost simultaneously) by Belov \cite{Be}, Grishin \cite{Gr}, and Shchigolev \cite{VS}. 

In \cite{AB} and \cite{IS} it was proved that the Specht problem has positive answer for graded associative algebras over fields of characteristic zero when the grading group is finite. 

Parts of the theory developed by Kemer do not work so well for non-associative algebras. Thus one should impose certain restrictions on the classes of algebras when studying the Specht problem for Lie or Jordan algebras. In characteristic 0, Iltyakov \cite{ilt} proved that if $L$ is a finitely generated Lie algebra and $A$ is an associative enveloping algebra for $L$ such that $A$ is PI then the weak polynomial identities for the pair $(A,L)$ are finitely based. A consequence of this result is that if $L$ is finitely generated and the adjoint Lie algebra $Ad(L)$ generates an associative PI algebra then the ideal of identities of the Lie algebra $L$ satisfies the Specht property. Recall that this is the case when $L$ is finite dimensional. Va\u{\i}s and Zelmanov \cite{vais_zelm} proved that if $J$ is a finitely generated Jordan PI-algebra, over a field of characteristic 0, then the ideal of identities of $J$ satisfies the Specht property. Once again Iltyakov \cite{ilt_alt} established the Specht property for the ideals of identities of finitely generated alternative algebras in characteristic 0. 

Therefore it is interesting to study the Specht problem for concrete varieties of Lie and Jordan algebras. The variety of Lie algebras generated by $sl_2(F)$, the simple 3-dimensional Lie algebra, in characteristic 0, satisfies the Specht property, this was proved by Razmyslov \cite{razm1, razm2}. Krasilnikov showed in \cite{Kr0} that the variety of Lie algebras defined by $UT_n^{(-)}(F)$, the Lie algebra of the $n\times n$ upper triangular matrices over a field $F$ of characteristic 0, satisfies the Specht property. Later on Krasilnikov in \cite{Kr} proved that if $F$ is infinite of characteristic $p>0$ and if the commutator ideal of a Lie algebra is $(p-1)$-nilpotent then the identities of that Lie algebra admit a finite basis. In this way he generalized the result from \cite{Kr0} to the case when $F$ is infinite of characteristic 0 or $p\ge n$. 

Now we recall some of the known results concerning group graded Lie and Jordan algebras and their identities. In \cite{pk_sl2} finite bases of the graded identities of $sl_2(F)$ were found, for an infinite field of characteristic different from 2, and an arbitrary grading. In \cite{GS} it was proved that, in characteristic 0, the variety of group graded Lie algebras generated by $sl_2$ has the Specht property.

In \cite{KM} and \cite{CM} the authors studied graded identities for  $UJ_2$, the Jordan algebra of $2\times 2$ upper triangular matrices. In \cite{CMS} it was shown that the variety generated by $UJ_2$ has the Specht property when it is graded by any finite abelian group. In \cite{KS} the graded identities for any $\mathbb{Z}_2$-grading on the Jordan algebra of symmetric matrices of order two were obtained, and in \cite{SS} the Specht property for the finite dimensional Jordan algebra of a non-degenerate symmetric bilinear form, graded by $\mathbb{Z}_2$, in characteristic 0, was established.  Here we recall that the more difficult situation where there is no grading at all, for this algebra, was settled by Iltyakov \cite{ilt_jordan} in the finite dimensional case. The infinite dimensional Jordan algebra of a non-degenerate symmetric bilinear form also satisfies the Specht property in characteristic 0, this follows by combining results obtained by Vasilovsky \cite{vas} and by Koshlukov \cite{pkja}.  Recently, in \cite{MS} it was shown that for any grading, the variety of graded commutative algebras generated by $(UT_2,\circ)$ has the Specht property in characteristic $2$.

We recall that if the grading group is infinite then, even in characteristic 0,  the graded identities of an algebra need not be finitely based, see for example \cite{fkk, fid_pk}.

In \cite{KY}, a finite basis for the graded identities for $UT_n(F)^{(-)}$ was found when this algebra is endowed with the canonical grading of $\mathbb{Z}_n$ and the field $F$ is infinite. 

In this paper, we study the variety of graded Lie algebras generated by  $UT_n(F)^{(-)}$, endowed whit the canonical $\mathbb{Z}_n$-grading. We prove that, when the characteristic of $F$ is 0 or is a prime $p\ge n$, it satisfies the Specht property. In order to achieve this we employ properties of partially  well-ordered sets. Furthermore we prove that the restriction $p\ge n$ for the characteristic of the base field cannot be removed. Namely we prove that the $\mathbb{Z}_3$-graded identities of $UT_3^{(-)}(F)$ do not satisfy the Specht property if $F$ is an infinite field of characteristic 2. To the best of our knowledge this is the first example of a finite dimensional Lie algebra nontrivially graded by a finite group that does not satisfy the Specht property.

The paper is organized as follows. In the Preliminaries we recall the required facts concerning polynomial identities, gradings, and partially well ordered sets. In Section 3 we deal first with the graded identities of the Lie algebra $UT_2^{(-)}(F)$, and prove that the corresponding ideal of graded identities satisfies the Specht property. Afterwards we consider $UT_3^{(-)}(F)$. Initially we require $F$ an infinite field of characteristic 0 or $p>2$, and prove that the ideal of graded identities satisfies the Specht property. We also show that if $F$ is of characteristic 2, then the Specht property fails in this case. In the last section we prove the Specht property for $UT_n^{(-)}(F)$ in case $\ch F=0$ or $\ch F\ge n$. We chose to separate the general case from those when $n\le 3$ for several reasons. One of them is that the case $n=2$ is much simpler and transparent. Another is that when $n=3$ we have two completely different situations: when $\ch F=2$ and when $\ch F\ne 2$. And the last reason is that the arguments in the cases $n=2$ and $n=3$ are more transparent and give a better idea of the methods used in the general case. 

\section{Preliminaries}
\subsection{Graded identities}
Let $G$ be a multiplicative group with unit element $1$. 
We fix an infinite field $F$ and we consider all algebras and vector spaces over $F$. Let A be a $G$-graded algebra (not necessarily associative), that is $A =\bigoplus_{g\in G} A_g$, a direct sum of vector subspaces such that $A_gA_h\subseteq A_{gh}$ for every $g$, $h\in G$. The elements $a\in A_g$ are \textsl{homogeneous} of 
degree $g$, and we denote this as $\Gdeg a = g$, or if not ambiguous, simply as $\deg a=g$.
\begin{defi}
Let $A=\bigoplus_{g\in G} A_g$ and $B=\bigoplus_{g\in G} B_g$ be two $G$-graded algebras. The map $f\colon A\to B$ is a homomorphism (endomorphism, automorphism, isomorphism) of $G$-graded algebras if $f$ is a homomorphism (endomorphism, automorphism, isomorphism) of algebras such that $f(A_g)\subseteq B_g$.

The algebras $A$ and $B$ are $G$-graded isomorphic if there exists an isomorphism of $G$-graded algebras  $f\colon A \to B$.
\end{defi}
Note that the previous definition says that $A$ and $B$ are isomorphic as $G$-graded algebras whenever there is an isomorphism of algebras from $A$ to $B$ which respects the gradings.
A vector subspace $B\subseteq A$ is called \textsl{graded} (or homogeneous) if $B=\bigoplus_{g\in G}(A_g\cap B)$. If $I\subseteq A$ is an ideal and a graded subspace, we call it a \textsl{graded ideal}. In this case the quotient $A/I$ inherits from A a natural structure of $G$-graded algebra.

We recall the definition of the \textsl{free $G$-graded algebra}. Let $g\in G$ and let $X_g= \{x_1^{(g)},x_2^{(g)},\ldots,\}$ be an infinite (countable) set of variables. Put $X=X_G=\bigcup_{g\in G} X_g$ and form the free algebra (associative, or Lie, or whatever) $F\{X_G\}$ in the variables of $X$.  We can define naturally a $G$-grading on $F\{X_G\}$ by assigning degree $g$ to all variables from $X_g$, $\deg x_i^{(g)} = g$, and then extending this to all monomials $m\in F\{X_G \}$ in the following way
\[
\deg m=\left\{ \begin{array}{rcl}
            g, & \text{ if }&  m=x_{i}^{(g)}\\     
         (\deg m_1)(\deg m_2) & \text{ if } & m=m_1m_2\end{array}
   \right.
   \]
In the case of associative and Lie algebras, we denote the corresponding free algebra by
$F\langle X_G\rangle$ and $ \mathcal{L}\langle X_G \rangle$, respectively.
\begin{defi}
Consider $A$ a $G$-graded algebra in the class of  $F\{X_G\}$ and let $f(x_1^{(g_1)},\ldots, x_n^{(g_{n})})\in F\{X_G\}$. Then $f$ is a $G$-graded polynomial identity for $A$ if $f(a_1 ,\ldots , a_m )= 0$ for all $a_1$,  \dots, $a_m\in A$ such that $\deg a_i= g_i$ for every $i$.
\end{defi}
We denote by $T_G(A)$ the set of $G$-graded identities for $A$. Observe that $T_G(A)$ is closed under endomorphisms of $F\{X_G\}$ that respect the grading. Conversely every such ideal is the ideal of $G$-graded identities for some $G$-graded algebra.

If $J\subseteq F\{X_G\}$ is a $G$-graded ideal that is closed under endomorphisms of $F\{X_G\}$, then $J$ is called a $T_{G}$-ideal.

We recall briefly some basic notions from PI theory we will need. Let $\{f_i\mid i\in I\}$ and $g$ be graded polynomials where $I$ is some set of indices, then $g$ is a \textsl{consequence} of $\{f_i\mid i\in I\}$ whenever $g$ lies in the least $T_G$-ideal that contains the set $\{f_i\mid i\in I\}$. Two sets of graded polynomials are \textsl{equivalent} if the $T_G$-ideals generated by these sets are equal. If $S$ is a set of graded polynomials we denote by $\langle S\rangle_{T_G}$ the $T_G$-ideal generated by $S$. 

\begin{defi}
Let $A$ be a $G$-graded algebra. Then $T_G(A)$ has the Specht property if every $T_{G}$-ideal $J$ such that  $T_G(A)\subseteq J$, has a finite basis, that is $J$ is finitely generated as a $T_{G}$-ideal. In particular, $T_G(A)$ must be finitely generated as a $T_G$-ideal.
\end{defi}

\subsection{Finite basis property for sets}
We collect notions concerning orders and the finite basis property for sets. For more details see \cite{Hig}.

Let $P$ be a non-empty set. A relation $p_1\leq p_2$ on $P$ is a \textsl{quasi-order} if it is reflexive and transitive. It means that $(i)$ $p_1\leq p_1$
for every $p_1\in P$, and $(ii)$ $p_1\leq p_2$ and $p_2\leq p_3$ imply $p_1\leq p_3$. If also $p_1\leq p_2$ and $p_2 \leq p_1$ imply $p_1 = p_2$, the relation is an \textit{order}; and if in addition for every $p_1$, $p_2$ either $p_1\leq p_2$ or $p_2\leq p_1$, it is a linear order.

If $Q$ is a subset of the quasi-ordered set $P$, the closure of $Q$, written $\overline{Q}$, is the set of all elements $p\in P$ such that for some $q$ in $Q$, $q\leq p$. A \textit{closed} subset of $P$ is one that coincides with its own closure. The quasi-ordered set $P$ has the \textsl{finite basis property} (f.b.p.), also called \textsl{well-quasi-ordered set}, if every closed subset of $P$ is the closure of a finite set of elements. The following theorem was proved in \cite{Hig} and gives useful equivalences to the f.b.p.
\begin{theo}[Higman \cite{Hig}]\label{eqhigman}
The following conditions on a quasi-ordered set $P$ are equivalent.
\begin{enumerate}[label=\roman*.]
\item Every closed subset of $P$ is the closure of a finite subset;
\item  If $Q$ is any subset of $P$, there is a finite $Q_0$ such that $Q_0\subseteq Q\subseteq \overline{Q}_0$ ;
\item Every infinite sequence of elements $\{p_i\}_{i\geq 1}$ of $P$ has an infinite ascending subsequence
$$p_{i_1}\leq p_{i_2}\leq\cdots \leq p_{i_k}\leq\cdots ;$$
\item There exists neither an infinite strictly descending sequence in $P$ nor an infinite one consisting of pairwise incomparable elements of $P$.
\end{enumerate}
\end{theo}

Consider $P$ a quasi-ordered set and let $D(P)$ be the set
of all finite sequences of elements of $P$. Then $D(P)$ is quasi-ordered by the rule: $x \leq y$ if $x$ is majorized by a subsequence of $y$. In other words  $x = (p_1,\ldots, p_n) \leq y = (q_1,\ldots, q_s)$ if there exists an order preserving injection $\varphi\colon \mathbb{N}\to\mathbb{N}$ such that $\varphi(n)\leq s$ and $p_i\leq q_{\varphi(i)}$ for every $i=1$, \dots, $n$.

As an example, take the set of positive integers $\mathbb{N}$, with respect to the usual order it is well ordered. The following theorem shows that $D(\mathbb{N})$ satisfies the f.b.p. (Clearly this fact can be deduced by an elementary and easy argument.)

\begin{theo}[Higman, \cite{Hig}]\label{DP}
Let $P$ be a quasi-ordered set. If $P$ has f.b.p., so has $D(P)$.
\end{theo}

The next result will be very useful in our paper.
\begin{prop}[Higman \cite{Hig}]\label{proordem}
Let $(P_1,\leq_1)$, $(P_2,\leq 2)$, \dots, $(P_k,\leq_k)$ be quasi-ordered sets satisfying the f.b.p.
\begin{enumerate}[label=\roman*.]
\item 
The disjoint union of $P_1$, $P_2$, \dots, $P_k$, endowed with the quasi-order where $p\leq q$ if and only if $p$, $q\in P_i$ and $p \leq_{i} q$ for some $i\in\{1,2,\ldots,k\}$, satisfies the f.b.p.
\item 
The Cartesian product $P_1\times P_2\times\cdots\times P_k$ endowed with the quasi-order given by  $(p_1,p_2,\ldots,p_k)\leq (q_1,q_2,\ldots, q_k)$ whenever $p_i\leq_i q_i$ for every $i\in\{1,\ldots,k\}$, satisfies the f.b.p.
\end{enumerate}
\end{prop}
\begin{exam}\label{odernk}
\textnormal{
Let $k$ be a positive integer and consider $\mathbb{N}$ endowed with its natural order $\leq$. Then by Proposition \ref{proordem},  
$\mathbb{N}^k$ is partially well ordered with the order
\[
(n_1,n_2,\ldots,n_k)\leq'_k (m_1,m_2,\ldots, m_k) \text{ if } n_i\leq m_i
\text{ for every } i\in\{1,\ldots,k\}.
\]
Theorem \ref{DP} then implies that $(D(\mathbb{N}^k),\leq_k)$ has f.b.p. where the order $\leq_k$ is defined in the following way:
$(p_1,\ldots, p_n) \leq_k (q_1,\ldots, q_s)$ if there exists an order preserving injection
$\varphi\colon \mathbb{N}\rightarrow\mathbb{N}$ such that $\varphi(n)\leq s$ and $p_i\leq'_k q_{\varphi(i)}$ for any $i=1$, \dots, $n$.}
\end{exam}
\section{The graded identities of $T_{\mathbb{Z}_3}(UT_3(F))$}

Let $F$ be an infinite field and let $L=UT_n(F)^{(-)}$ be the Lie algebra of the $n\times n$ upper triangular matrices over $F$. If $a$, $b\in L$, denote  the commutator $[a,b]=ab-ba$. The Lie brackets are asssumed left normed, that is, $[a,b,c]=[[a,b],c]$. Denote by $e_{ij}$  the matrix units, with 1  at position $(i,j)$ and 0 elsewhere. 

We consider a particular but important grading on $L$.
Fix $G=\mathbb{Z}_n$ in additive notation, then the algebra $L$ is $G$-graded by setting $L=\oplus_{k=0}^{n-1}L_k$ where $L_k$ is the span of all $e_{ij}$ such that $j-i=k$. Recall that $e_{ij}$ is the matrix unit that has 1 at position $(i,j)$ and zero elsewhere. This grading is called the \textsl{canonical $\mathbb{Z}_n$-grading}. The main goal of this and the following section is to prove that the $T_{\mathbb{Z}_n}$-ideal of graded identities of $L$ has the Specht property when $\ch F\geq n$ (or $\ch F=0$).

The $T_{\mathbb{Z}_n}$-ideal of graded identities of $L$ over an infinite field was described in \cite{KY}.
\begin{theo}[\cite{KY}]\label{idutnf}
The $\mathbb{Z}_n$-graded identities of the Lie algebra $L$ of the upper triangular $n\times n$ matrices over $F$ follow from
\begin{equation}
\left\{ \begin{array}{cc}
             [x_1^{(i)},x_2^{(j)}], &   i+j\geq n \\     
            {[}x_1^{(0)}, x_2^{(0)}{]} \end{array}
   \right.
\end{equation}
\end{theo}

\subsection{The case $UT_2(F)^{(-)}$}
In this short subsection we treat the case $n=2$. Let $Y$ and $Z$ be two infinite countable sets, $Y=\{y_1, y_2,\ldots \}$ and $Z=\{z_1,z_2,\ldots\}$. The free Lie algebra $\mathcal{L}(Y\cup Z)$ freely generated over $F$ by $Y\cup Z$ has a natural $\mathbb{Z}_2$-grading assuming the variables in $Y$ to be of degree  $0$ and those in $Z$ of degree $1$. Consider the algebra $UT_2(F)^{(-)}$  endowed with the canonical $\mathbb{Z}_2$-grading, and denote by $I$ the ideal  of $\mathbb{Z}_{2}$-graded identities for $UT_2(F)$.  By  Theorem \ref{idutnf} and the Jacobi identity, we get
\begin{equation}\label{ycomutan}
[z,y_1,y_2]-[z,y_2,y_1]\in I,
\end{equation}
hence the non-zero monomials in $L(Y\cup Z)/I$ are of the type
\[
[z,y_{i_1},y_{i_2},\ldots,y_{i_k}]
\]
where $i_1\leq i_2\leq\cdots\leq i_k$.

We denote $[z,\underbrace{y_1,\ldots,y_1}_{a_1},\ldots,\underbrace{y_n,\ldots,y_n}_{a_n}]$ by $[z,a_1y_1,\ldots,a_ny_n]$

Consider the following set
\[
\mathcal{B}=\{[z,a_1y_1,\ldots,a_ny_n]| n\geq 1,\ a_i>0, \text{ for every $i$}\}.
\]
Given $f=[z,a_1y_1,\ldots,a_ny_n] \in\mathcal{B}$, we define $V_f=(a_1,\ldots, a_n)$. Observe that $\{V_f\mid f\in \mathcal{B}\} \subseteq D(\mathbb{N})$, hence $V_f=(a_1,\ldots, a_n)\leq V_g=(a'_1,\ldots, a'_m)$ if there exists an order preserving injection $\varphi\colon \mathbb{N}\to\mathbb{N}$ such that $\varphi(n)\leq m$ and $a_i\leq a'_{\varphi(i)}$ for every $i=1$, \dots, $n$.
\begin{prop}\label{conseut2}
Let $f$, $g\in\mathcal{B}$. If $V_f\leq V_g$ then $g$ is a consequence of $f$ modulo $I$, that is $g\in\langle f\cup I\rangle_{T_{\mathbb{Z}_2}}$.
\end{prop}
\begin{proof}
Suppose that $V_f=(a_1,\ldots, a_n)$ and  $V_g=(a'_1,\ldots, a'_m)$. By hypothesis, there is a subsequence $(a'_{i_1},\ldots, a'_{i_n})$ of $V_g$ such that $a_j\leq a'_{i_j}$ for every $j\in\{1,\ldots,n\}$.
Let $f(z,y_{i_1},\ldots, y_{i_n})$ be the polynomial obtained by replacing the variable $y_j$ by $y_{i_j}$ for each $1\leq j\leq n$ in $f(z,y_1,\ldots,y_n)$. Then by equality (\ref{ycomutan}), we get 
\[
[f(z,y_{i_1}, \ldots, y_{i_n}), (a'_{i_1}-a_1)y_{i_1},\ldots,(a' _{i_n}-a_n)y_{i_n}]\equiv [z,a'_{i_1}y_{i_1},\ldots, a'_{i_n}y_{i_n}] \pmod{I}.
\]
If $\{l_1,\ldots, l_{m-n}\}=\{1,\ldots,m\}\setminus\{i_1,\ldots,i_m\}$ by using once again (\ref{ycomutan}), we get
\[
[z,a'_{i_1}y_{i_1},\ldots, a'_{i_n}y_{i_n},a'_{l_1}y_{l_1},\ldots, a'_{l_{m-n}}y_{l_{m-n}} ]\equiv g \pmod{I}.\qedhere
\]
\end{proof}

\begin{defi}\label{oderinB}
We define a quasi-order in $\mathcal{B}$ by $f\leq_{\mathcal{B}} g$  if $V_f\leq V_g$, for $f$, $g\in\mathcal{B}$.
\end{defi}
\begin{prop}\label{Ut2fbp}
The set $(\mathcal{B}, \leq_{\mathcal{B}})$ satisfies the f.b.p.
\end{prop}
\begin{proof}
Suppose that there is an infinite sequence $\{f_i\}_{i\geq 0}$ of  pairwise incomparable elements in $\mathcal{B}$ with respect to the order $\leq_{\mathcal{B}}$. This sequence leads to the sequence $\{V_{f_i}\}_{i\geq 0}$ in $D(\mathbb{N})$, by definition \ref{oderinB}. The elements of the sequence $\{V_{f_i}\}_{i\geq 0}$ are pairwise incomparable, but this contradicts Theorem \ref{eqhigman} because $D(\mathbb{N})$ is partially-well ordered.
\end{proof}
\begin{theo}
Let $J$ be a $T_{\mathbb{Z}_2}$-ideal such that $I\subseteq J$, then $J$ is finitely generated as a $T_{\mathbb{Z}_2}$-ideal.
\end{theo}
\begin{proof}
Since $F$  is an infinite field, we know that $J$ is generated as $T_{\mathbb{Z}_2}$-ideal by its multihomogeneous polynomials. Hence there exists a subset $\mathcal{B}'\subseteq\mathcal{B}$ such that 
\[
J=\langle \mathcal{B}'\rangle_{T_{\mathbb{Z}_2}} \pmod{I}.
\]
By Proposition \ref{Ut2fbp} and Theorem \ref{eqhigman}, there exists $\mathcal{B}_0\subseteq \mathcal{B}'$ such that $\mathcal{B}_{0}$ is finite and $\mathcal{B}_{0}\subseteq \mathcal{B}'\subseteq\overline{\mathcal{B}}_0$. 
It follows that given $g\in\mathcal{B}'$, there exists $f\in\mathcal{B}_0$ such that $V_f\leq V_g$. By Proposition \ref{conseut2}, $g\in\langle f\rangle_{T_{\mathbb{Z}_2}}\subseteq \langle\mathcal{B}_0\rangle_{T_{\mathbb{Z}_2}}\pmod{I}$, then
\[
J= \langle \mathcal{B}_0\rangle_{T_{\mathbb{Z}_2}} \pmod{I}. 
\]
Hence the $T_{\mathbb{Z}_2}$-ideal $J$ is finitely generated.
\end{proof}

\subsection{The case $UT_3(F)^{(-)}$}

Now we consider separately the case of $3\times 3$ upper triangular matrices. 
Let $Y$, $Z$ and $W$ be infinite countable sets, namely $Y=\{y_1, y_2,\ldots \}$, $Z=\{z_1,z_2,\ldots\}$ and $W=\{w_1,w_2,\ldots\}$. Consider the free Lie algebra $\mathcal{L}(Y\cup Z\cup W)$ generated over $F$ by $Y\cup Z\cup W$. Then $\mathcal{L}(Y\cup Z\cup W)$ has a natural $\mathbb{Z}_3$-grading assuming the variables $Y$, $Z$ and $W$ to be of degree 0, 1 and 2, respectively. The algebra $UT_3(F)$ is endowed with the canonical $\mathbb{Z}_3$-grading, and we denote by $I$ the ideal  of $\mathbb{Z}_{3}$-graded identities for $UT_3(F)$.

By  Theorem \ref{idutnf}, the  $\mathbb{Z}_{3}$-graded identities for $UT_3(F)$ follow from 
\begin{align}
             &[y_1,y_2],  \nonumber  \\  
\label{idut3} &\textcolor{red}{[z_1,z_2,z_3],}\\
             &{[}w_1,w_2{]}, \nonumber  \\
            &{[}z,w{]}\nonumber.  
            \end{align}
Using the previous identities and the Jacobi identity, we have the following equalities modulo $I$.
\begin{align}
\label{zcomutanUT3}
[z,y_1,y_2]&=[z,y_2,y_1],\\
\label{wcomutanUT3}
[w,y_1,y_2]&=[w,y_2,y_1],\\
\label{zwcomutanUT3}
[z_1,y_1,y_2, z_2 ]&=[z_1,y_2,y_1, z_2].
\end{align}
Observe that (\ref{zwcomutanUT3}) follows from (\ref{zcomutanUT3}); as we shall need frequently the former equality we chose to state it explicitly. 
Hence the non-zero monomials in $\mathcal{L}(Y\cup Z\cup W)/I$ are of the following types:
\begin{enumerate}[label=\Roman*.]
    \item $[z,y_{i_1},\ldots, y_{i_k}]$,
    \item  $[w,y_{i_1},\ldots, y_{i_k}]$,
    \item  $[z,y_{i_1},\ldots, y_{i_k}, z, y_{j_1},\ldots, y_{j_l}]$,
    \item   $[z_1,y_{i_1},\ldots, y_{i_k}, z_2, y_{j_1},\ldots, y_{j_l}]$,
\end{enumerate}
where $i_1\leq i_2\leq\cdots\leq i_k$ and $j_1\leq i_2\leq\cdots\leq j_l$.
Observe that if $f$ is a monomial of type III or IV, then its first and/or second block of variables $Y$ may be  empty

Let $f\in \mathcal{L}(Y\cup Z\cup W)/I$ be a multihomogeneous polynomial, then $f$ is equivalent to a monomial of type I or II or $f$ is a linear combinations of monomials of  type III or type IV.\\

\begin{notation}
\begin{align*}
[z,a_1y_1,\ldots,a_ny_n]&= [z,\underbrace{y_1,\ldots,y_1}_{a_1},\ldots,\underbrace{y_n,\ldots,y_n}_{a_n}]\\
[w,a_1y_1,\ldots,a_ny_n]&= [w,\underbrace{y_1,\ldots,y_1}_{a_1},\ldots,\underbrace{y_n,\ldots,y_n}_{a_n}]\\
[z_1,a_1y_1,\ldots,a_ny_n, z_2, b_1y_1,\ldots,b_ny_n]& 
=[z_1,\underbrace{y_1,\ldots,y_1}_{a_1},\ldots,\underbrace{y_n,\ldots,y_n}_{a_n}, z_2, \\
&\underbrace{y_1,\ldots,y_1}_{b_1},\ldots, \underbrace{y_n,\ldots, y_n}_{b_n}]
\end{align*}
\end{notation}
\indent
Consider the following sets
\begin{align*}
\mathcal{B}_1&=\{[z,a_1y_1,\ldots,a_ny_n]\mid n\geq 1,\ a_i> 0, \text{ for all $i$}\},\\
\mathcal{B}_2&=\{[w,a_1y_1,\ldots,a_ny_n]\mid n\geq 1,\ a_i>0,\ \text{ for all $i$}\}, \\
\mathcal{B}_{(1,1)}&=\{[z_1,a_1y_1,\ldots,a_ny_n, z_2, b_1y_1,\ldots,b_ny_n]\mid n\geq 0,\ a_i,b_i\geq 0,\ \text{ for all $i$}\}.
\end{align*}
Given $f\in\mathcal{B}_i$, we assign to it the finite sequence $V_f=(a_1,\ldots, a_n)$ where $1\leq i\leq 2$. Note that $V_f\in D(\mathbb{N})$. So, we define the following order in $f\in\mathcal{B}_i$.
\begin{defi}
Let $f$, $g\in\mathcal{B}_i$ where $i=1$ or 2, we define $f\leq_{\mathcal{B}_i} g$ whenever $V_f\leq V_g$.
\end{defi}

\begin{prop}\label{fbporUT3}
Consider the set $\mathcal{B}_i$ and let $f$, $g\in  \mathcal{B}_i$.
\begin{enumerate}[label=\roman*.]
\item If $f\leq_{\mathcal{B}_i} g$, then $g$ is a consequence of $f$ modulo $I$.
\item $(\mathcal{B}_i, \leq_{\mathcal{B}_i})$ satisfies the f.b.p.
\end{enumerate}
\end{prop}

\begin{proof}
The argument is analogous to the proofs of Propositions \ref{conseut2} and \ref{Ut2fbp}.
\end{proof}

\begin{cor}\label{Bifinito}
Let $J$ be a $T_{\mathbb{Z}_3}$-ideal such that $I\subseteq J$ and let 
$\mathcal{A}_{i}=\{f\in J\mid f\in \mathcal{B}_{i}\}$, $i=1$, 2.
Then there exist finite subsets $\mathcal{A}' _{i}\subseteq \mathcal{A}_{i}$ such that 
\[
\langle \mathcal{A}_{i}\rangle_{T_{\mathbb{Z}_3}}= \langle \mathcal{A}'_{i}\rangle_{T_{\mathbb{Z}_3}}\pmod{I}.
\]
\end{cor}
\begin{proof}
By item ii. of Propostion \ref{fbporUT3} and Theorem \ref{eqhigman}, there exist finite subsets $\mathcal{A}'_{i}\subseteq \mathcal{A}_{i}$ where $i=1$, 2, such that 
\[
\mathcal{A}'_{i}\subseteq \mathcal{A}_{i}\subseteq\overline{\mathcal{A}'_{i}}.
\]
It follows that given $g\in\mathcal{A}_{i}$, there exists $f\in \mathcal{A}' _{i}$ such that $f\leq_{\mathcal{B}_i} g$. Hence, by item i. of Proposition \ref{fbporUT3}, $g$ is a consequence of $f$ modulo $I$. Then
\[
\langle \mathcal{A}_{i}\rangle_{T_{\mathbb{Z}_3}}= \langle \mathcal{A}'_{i}\rangle_{T_{\mathbb{Z}_3}} \pmod{I}. \qedhere
\]
\end{proof}
If $f\in\mathcal{B}_{(1,1)}$, we assign to it the finite sequence $V_f=((a_1,b_1),\ldots, (a_n,b_n))$. Observe that $V_f\in D({\mathbb{N}_{0}}^2)$.
Recall that  by Theorem \ref{DP} and Proposition \ref{proordem}, the set $D({\mathbb{N}_{0}}^2)$ is  partially well ordered.
\begin{defi}\label{qorderUT3}
Let $f$, $g$ be two polynomials in $\mathcal{B}_{(1,1)}$. Define the following quasi-order in $\mathcal{B}_{(1,1)}$: $f\leq_{\mathcal{B}_{(1,1)}} g$ whenever $V_f\leq_2 V_g$.
\end{defi}
\begin{prop}\label{conseUT3}
If $f\leq_{\mathcal{B}_{(1,1)}} g$, then $g\in\langle f\cup g \rangle_{T_{\mathbb{Z}_3}}$.
\end{prop}
\begin{proof}
Suppose that $V_f=((a_1,b_1),\ldots, (a_n,b_n))$ and $V_g=((a'_1,b'_1),\ldots, (a'_m,b'_m))$. 
By hypothesis there exists a subsequence $((a'_{i_1},b'_{i_1}),\ldots, (a'_{i_n},b'_{i_n}))$ of $V_g$ such that $a_k\leq a'_{i_k}$ and $b_k\leq b'_{i_k}$ for every $k\in\{1,\ldots,n\}$. Let $f(z_1,z_2,y_{i_1},\ldots, y_{i_n})$  be the polynomial obtained replacing the variable $y_k$ by $y_{i_k}$ for each $1\leq k\leq n$ in $f(z_1, z_2, y_1,\ldots,y_n)$. Then by equality (\ref{wcomutanUT3}), we get
\begin{align*}
f_1&= [f(z_1, z_2,y_{i_1},\ldots, y_{i_n}), (b'_{i_1}-b_1)y_{i_1},\ldots,(b' _{i_n}-b_n)y_{i_n}]\\
&\equiv [z_1,a_{i_1}y_{i_1},\ldots, a_{i_n}y_{i_n},z_2,b'_{i_1}y_{i_1},\ldots, b'_{i_n}y_{i_n}]  \pmod{I}.
\end{align*}
Replacing the variable $z_1$ by $[z_1,(a'_{i_1}-a_1)y_{i_1},\ldots,(a' _{i_n}-a_n)y_{i_n}]$ in $f_1$ and applying several times (\ref{zcomutanUT3}), we obtain
\begin{align*}
&f_1([z_1,(a'_{i_1}-a_1)y_{i_1},\ldots,(a' _{i_n}-a_n)y_{i_n}], z_2,y_{i_1},\ldots, y_{i_n})\\
&\equiv [z_1,a'_{i_1}y_{i_1},\ldots, a'_{i_n}y_{i_n},z_2,b'_{i_1}y_{i_1},\ldots, b'_{i_n}y_{i_n}].
\end{align*}
Let $\{l_1,\ldots, l_{m-n}\}=\{1\ldots,m\}\setminus \{i_1,\ldots,i_m\}$.
It follows from  equalities (\ref{zcomutanUT3}) and (\ref{wcomutanUT3})
\begin{align*}
[[z_1,a'_{l_1}y_{l_1},& \ldots, a'_{l_{m-n}}y_{l_{m-n}}],a'_{i_1}y_{i_1},\ldots, a'_{i_n}y_{i_n},z_2,\\
&b'_{i_1}y_{i_1},\ldots, b'_{i_n}y_{i_n}, b'_{l_1}{y}_{l_1},\ldots, b'_{l_{m-n}}y_{l_{m-n}}] \equiv g \pmod{I}.
\end{align*}
Hence $g$ is a consequence of $f$ modulo $I$.
\end{proof}
\begin{prop}\label{wellqorderUT3}
$(\mathcal{B}_{(1,1)},\leq_{\mathcal{B}_{(1,1)}})$ is a quasi well ordered set.
\end{prop}
\begin{proof}
Suppose that there is an infinite sequence $\{f_i\}_{i\geq 0}$ of  pairwise incomparable elements in ${\mathcal{B}}_{(1,1)}$ with respect to the order $\leq_{\mathcal{B}_{(1,1)}}$. The above sequence defines the infinite sequence $\{V_{f_i}\}_{i\geq 0}$ in $D({\mathbb{N}_{0}}^2)$. Note that the elements of the sequence $\{V_{f_i}\}_{i\geq 0}$ are pairwise incomparable, but this is a contradiction since the set $D({\mathbb{N}_{0}}^2)$ is-partially well ordered.
\end{proof}

\begin{lem}\label{liut3}
The commutators 
\[
c=[z_1,y_1,\ldots, y_t,z_2, y_{t+1},\ldots, y_{t+k}]
\]
are linearly independent modulo the $\mathbb{Z}_3$-graded identities for $UT_3(F)^{(-)}$.
\end{lem}
\begin{proof}
Here we use a technique based on generic matrices, adapted to our case. First, consider the substitution $z_1=e_{12}$ and $z_2=e_{23}$. Computing $c$ with the above substitution of the $z_i$, and assuming $y_i= y_i^1e_{11}+y_i^2 e_{22}++y_i^3 e_{33}$ where the $y_i^j$ are commuting independent variables (here $j$ is an upper index, not an exponent), yields
\[
c=\prod_{s=1}^2 \prod_i (y_i^{s+1} - y_i^1) e_{13}.
\]
Here in the second product, $i$ ranges from $1$ to $t$, if $s=1$.  If $s=2$, $i$ ranges from $t+1$ to $t+k$.
Now make the following substitutions for the $y_i$.

First put all of the $y_i^j=0$ \textbf{except for}:
\begin{itemize}
\item
$i=1$, \dots, $t$: here define $y_i^{2}=1$.
\item
$i=t+1$, \dots, $t+k$: here define $y_i^{3}=1$.
\end{itemize}
Such a substitution vanishes all commutators but $c$, and $c=e_{13}$. 
So if we suppose there is a linear combination among commutators of the type $c$ (we assume them multihomogeneous; this is no loss of generality since the base field is infinite), we assume that $c$ participates with a nonzero coefficient in it. Then make the substitution for the $y_i$ as above. All commutators vanish except for $c$. This proves they are linearly independent.
\end{proof}
\begin{lem}\label{liut3h}
The commutators
$$c=[z_1, a_1y_1,\ldots, a_ny_n, z_2, b_1y_1,\ldots, b_ny_n]$$
are linearly independent modulo the $\mathbb{Z}_3$-graded identities for $UT_3(F)^{(-)}$.
\end{lem}
\begin{proof}
First, consider the substitution $z_1=e_{12}$ and $z_2=e_{23}$. Computing $c$ with the above substitution of the $z_i$, and assuming $y_i= y_i^1e_{11}+y_i^2 e_{22}++y_i^3 e_{33}$ where the $y_i^j$ are commuting independent variables as above, gives
\[
c= \left(\prod_{i=1}^n (y_i^{2} - y_i^1)^{a_i} \prod_{i=1}^n (y_i^{3} - y_i^1)^{b_i}\right)  e_{13}. 
\]
Now making $y_i^1=0$ for all $i$, we obtain 
\begin{equation}\label{eqc}
c=\left( \prod_{i=1}^n (y_i^{2})^{a_i}\prod_{i=1}^n (y_i^{3})^{b_i}\right)e_{13}. 
\end{equation}
Define the following monomial
\[
m_c=\prod_{i=1}^n (y_i^{2})^{a_i}\prod_{i=1}^n (y_i^{3})^{b_i}.
\]
Consider $c$, $c'$ two different commutators  and notice that  $m_c\neq m_{c'}$. Hence, the elements of the set
$\{ m_c\mid c\in\mathcal{B}_{(1,1)}\}$
are linearly independent. So if we suppose there is a nontrivial linear 
combination among commutators of type $c$ (we assume them multihomogeneous) such that
\[
\sum_{i=1}^t \alpha_i c_i(z_1,z_2,y_1,\ldots, y_n)\in I.
\]
By equality (\ref{eqc}), we have
\[
0=\left( \sum\limits_{j=1}^t \alpha_{j} m_{c_j} \right) e_{13}.
\]
It follows that 
\[
0=\sum\limits_{j=1}^t \alpha_{j} m_{c_j},
\]
hence $\alpha_{j}=0$. This proves the commutators $c$ are linearly independent.
\end{proof}

\begin{defi}\label{linordut3}
Let $f_i$, $f_j\in\mathcal{B}_{(1,1)}$ be multihomogeneous  polynomials of the same multidegree. Consider the finite sequences
\[
V_{f_i}= ((a_1,b_1),\ldots,(a_n,b_n)), \quad
V_{f_j}= ((a'_1,b'_1),\ldots,(a'_n,b'_n)).
\]
We define the order $f_j\leq'  f_i$ as follows: 
\[
f_j\leq'  f_i \text{ if } \sum_{i=1}^n a'_i> \sum_{i=1}^n a_i, \text{ or if }\sum_{i=1}^n a'_i=\sum_{i=1}^n a_i \text{ and }(a_1,\ldots, a_n)\leq_{\textrm{lex}} (a'_1,\ldots, a' _n).
\]
\end{defi}
The order $\leq'$ is linear on the polynomials of the same multidegree in $\mathcal{B}_{(1,1)}$.

\begin{defi}\label{mlUT3}
Let $f$ be a multihomogeneous polynomial such that 
\[
f=\sum\limits_{i=1}^n \alpha_i f_i,
\]
where $f_i\in\mathcal{B}_{(1,1)}$ and $\alpha_i\in F\setminus\{0\}$. We define the leading monomial of $f$ by
\[
\textrm{ml}(f)= \max_{\leq'} \{f_i\mid 1\leq i\leq n\},
\]
and the leading coefficient of $f$, denoted by $\textrm{cl}(f)$, as the coefficient of $\textrm{ml}(f)$. 
\end{defi}
By Lemmas \ref{liut3} and \ref{liut3h},  this way of writing $f$ as a linear combination of elements of $\mathcal{B}_{(1,1)}$ is unique. For this reason, we can define the leading monomial of $f$ with respect to the order $\leq'$.
\begin{prop}\label{orderpreservadout3}
Let $f(z_1,z_2,y_1,\ldots,y_n)$ be a multihomogeneous polynomial such that 
\[
f(z_1,z_2,y_1,\ldots,y_n)=\sum_{i=1}^n \alpha_i f_i(z_1,z_2,y_1,\ldots,y_n),
\]
where $f_i(z_1,z_2,y_1,\ldots,y_n)\in\mathcal{B}_{(1,1)}$ and $\alpha_i\in F\setminus\{0\}$. The following statements hold.
\begin{enumerate}[label=\roman*.]
\item If $g(z_1,z_2,y_1,\ldots,y_n,y)= [f(z_1,z_2,y_1,\ldots,y_n), y]$, then $ml(g)= [ml(f),y]$.
\item If $h(z_1,z_2,y_1,\ldots,y_n,y)= f([z_1,y],z_2,y_1,\ldots,y_n)$, then 
\[
ml(h)= ml(f)([z_1,y],z_2,y_1,\ldots,y_n).
\]
\end{enumerate}
\end{prop}
\begin{proof}
It follows from Definitions \ref{linordut3} and \ref{mlUT3}.
\end{proof}

\begin{prop}\label{seq2orderut3}
Let $f$, $g$ be two multihomogeneous polynomials such that
\[
f=\sum_{i=1}^t \alpha_i f_i,\qquad g=\sum_{j=1}^s \beta_j g_j,
\]
where $\alpha_i$, $\beta_j\in F\setminus\{0\}$, $f_i$, $g_j\in\mathcal{B}_{(1,1)}$ for all $1\leq i\leq t$ and $1\leq j\leq s$. Suppose that $ml(f)\leq_{\mathcal{B}_{(1,1)}} ml(g)$, then there exists $h\in\langle f\rangle_{\mathbb{Z}_3}$ modulo $I$ such that $ml(h)=ml(g)$ and $cl(h)=cl(f)$. 
\end{prop}
\begin{proof}
By proposition \ref{conseUT3}, we know that $ml(g)\in\langle ml(f)\rangle_{\mathbb{Z}_3}$. Then, making the same computations done on $ml(f)$ to obtain $ml(g)$, we can obtain a consequence $h$ from $f$. Moreover, by Proposition \ref{orderpreservadout3}, $ml(h)=ml(g)$. It is clear that the polynomial $h$ has the same leading coefficient as $f$.
\end{proof}
\begin{defi}
Let $f$ be a multihomogeneous polynomial that is a linear combination of polynomials in $\mathcal{B}_{(1,1)}$. Then $f$ is called polynomial of type $(1,1)$.
\end{defi}
\begin{prop}\label{sequt3}
There is no infinite sequence of polynomial  $\{f_i\}_{i\geq 1}$ of type $(1,1)$ such that
\[
f_i\notin \langle f_1,\ldots, f_{i-1}\rangle_{T_{\mathbb{Z}_3}} \pmod{I}
\]
where $i\geq 2$.
\end{prop}
\begin{proof}
Suppose, on the contrary, that there exists such an infinite sequence of polynomials $\{f_i\}_{i\geq 1}$. Moreover, suppose that the $f_i$'s are of different multidegrees in the variables $Y$. Define the following sets
\begin{itemize}
    \item $J_i=\langle f_1, \ldots, f_i\rangle_{T_{\mathbb{Z}_3}} \pmod{I}$;
    \item $R_i=\{f\in J_i\setminus J_{i-1}\mid f$ is of type $(1,1)$ and of the same multidegree in the variables $Y$  as $f_i\}$ ;
    \item $\widehat{R_i}=\{ ml(f)\mid f\in R_i\}$.
\end{itemize}
Note that $f_i\in R_i$, so without loss of generality, we can suppose that
\[
ml(f_i)=\min_{\leq'} \widehat{R_i}.
\]
We denote $m_i:= ml(f_i)$ where $i\geq 1$.
Then we have the infinite sequence $\{m_i\}_{i\geq 1}$ in $\mathcal{B}_{(1,1)}$. By Proposition \ref{wellqorderUT3}, we have that  $(\mathcal{B}_{(1,1)}, \leq_{\mathcal{B}_{(1,1)}})$ is a quasi well ordered set. It follows from Theorem \ref{eqhigman} that there exists an infinite subsequence $\{m_{i_k}\}_{k\geq 1}$ of the sequence $\{m_i\}_{i\geq 1}$ such that
\[
m_{i_1} \leq_{\mathcal{B}_{(1,1)}}m_{i_2} \leq_{\mathcal{B}_{(1,1)}}\cdots \leq_{\mathcal{B}_{(1,1)}} m_{i_k} \leq_{\mathcal{B}_{(1,1)}}\cdots
\]
where $i_1<i_2<\cdots i_k<\cdots$.

Let $\alpha_{i_k}$ be the leading coefficient of $f_{i_k}$, where $k\geq 1$. Take $s> 1$ such that
\[
\sum_{l=1}^s \alpha_{i_l}\neq 0.
\]
Recall that  $m_{i_l}\leq_{\mathcal{B}_{(1,1)}} m_{i_{s+1}}$,
where $l\in\{1,\ldots s\}$. By Proposition \ref{seq2orderut3}, there exists $h_l\in\langle f_{i_l}\rangle_{T_{\mathbb{Z}_2}}$ such that
\[
ml(h_l)= m_{i_{s+1}} \text{ and } cl(h_l)=\alpha_{i_l},
\]
for every $1\leq l\leq s$. Consider the polynomial
\[
h=\sum_{l=1}^s h_{l},
\]
and notice that $ml(h)=m_{i_{s+1}}$ and $cl(h)=\sum\limits_{l^1}^s \alpha_{i_l}$.

Observe that $h_{l}\in J_{i_l}\subseteq J_{i_{s+1}-1}$, where $1\leq l\leq s$, and since $J_{i_{s+1}-1}$ is $T_{\mathbb{Z}_2}$-ideal, we have $h\in J_{i_{s+1}-1}$. Define
\[
g= f_{i_{s+1}}- (\alpha_{i_{s+1}} cl(h)^{-1})h.
\]
Then $ml(g)\neq m_{i_{s+1}}$ and $ml(g)\leq' m_{i_{s+1}}$. Note that $g\in J_{i_{s+1}}$ because $f_{i_{s+1}}\in J_{i_{s+1}}$ and $h\in J_{i_{s+1}-1}\subseteq J_{i_{s+1}}$. On the other hand, $g\notin J_{i_{s+1}-1}$ because $f_{i_{s+1}}\notin J_{i_{s+1}-1}$ and $h\in J_{i_{s+1}-1}$. It follows that $g\in R_{i_{s+1}}$, then  $ml(g)\in\widehat{R}_{i_{s+1}}$, but $ml(g)<'m_{i_{s+1}}$ and $m_{i_{s+1}}=\min_{\leq'} \widehat{R}_{i_{s+1}}$, which is a contradiction.
\end{proof}
As a direct consequence of Proposition \ref{sequt3}, we have the following corollary.
\begin{cor}\label{B11finito}
Let $J$ be a $T_{\mathbb{Z}_3}$-ideal such that $I\subseteq J$. Consider the following set 
\[
\mathcal{A}_{(1,1)}=\{f\in J\mid f\ \textrm{is a polynomial of type $(1,1)$}\}.
\]
Then there exists a finite subset $\mathcal{A}' _{(1,1)}\subseteq \mathcal{A}_{(1,1)}$ such that 
\[
\langle\mathcal{A}_{(1,1)}\rangle_{T_{\mathbb{Z}_3}}= \langle \mathcal{A}'_{(1,1)}\rangle_{T_{\mathbb{Z}_3}} \pmod{I}.
\]
\end{cor}
Now we deduce the Specht property for the ideal of $\mathbb{Z}_3$-graded identities for $UT_3^{(-)}(F)$ when $F$ is an infinite field of characteristic different from 2.
\begin{theo}
Suppose that $\ch F\ne 2$. If $J$ is a $T_{\mathbb{Z}_3}$-ideal such that $I\subseteq J$ then $J$ is finitely generated as $T_{\mathbb{Z}_3}$-ideal.
\end{theo}
\begin{proof}
Since $F$ is infinite, $J$ is generated by its multihomogeneous polynomials. As $\ch F\ne 2$, using the multilinearization process, we can consider that every multihomogeneous polynomial with variables $z's$ is linear in these variables, because by the identities (\ref{idut3}), they appear in the non-zero monomials of  $\mathcal{L}(Y\cup Z\cup W)/I$ at most twice. Hence, $J$ is generated as a $T_{\mathbb{Z}_3}$-ideal, modulo $I$, by the following sets
\begin{itemize}
    \item  $\mathcal{A}_{i}=\{f\in J\mid f\in \mathcal{B}_{i}\}$ where $1\leq i\leq 2$;
    \item $\mathcal{A}_{(1,1)}=\{f\in J\mid f\ \textrm{is a polynomial of type $(1,1)$}\}$.
\end{itemize}
Using Corollaries \ref{Bifinito} and \ref{B11finito}, we get that there exist finite subsets $\mathcal{A}'_{i}\subseteq \mathcal{A}_{i}$, where $i=1$, 2, and $\mathcal{A}'_{(1,1)}\subseteq\mathcal{A}_{(1,1)}$ such that
\begin{align*}
\langle\mathcal{A}_{i}\rangle_{T_{\mathbb{Z}_3}}&= \langle \mathcal{A}'_{i}\rangle_{T_{\mathbb{Z}_3}} \pmod{I};\\
\langle\mathcal{A}_{(1,1)}\rangle_{T_{\mathbb{Z}_3}}&= \langle \mathcal{A}'_{(1,1)}\rangle_{T_{\mathbb{Z}_3}} \pmod{I}.
\end{align*}
It follows 
\[
J=\langle \mathcal{A}'_{1}\cup \mathcal{A}'_{2}\cup \mathcal{A}'_{(1,1)}\rangle_{T_{\mathbb{Z}_3}} \pmod{I},
\]
therefore $J= \langle \mathcal{A}'_{1}\cup \mathcal{A}'_{2}\cup \mathcal{A}'_{(1,1)}\cup I\rangle_{T_{\mathbb{Z}_3}}$, and $J$ is finitely generated as a $T_{\mathbb{Z}_3}$-ideal.
\end{proof}

\subsection{The case $UT_3^{(-)}(F)$ in characteristic 2}
In this subsection we prove that the graded identities of $UT_3^{(-)}(F)$ do not satisfy the Specht property if $F$ is an infinite field of characteristic 2. 
\begin{lem}\label{conseck}
Let $F$ be an infinite field of characteristic $2$. For $k\geq 1$, define the polynomial 
\[
c_k=[z,y_1,\ldots, y_k, z]. 
\]
Consider $f$ a consequence of $c_k \pmod{I}$ such that $f$ is a multihomogeneous polynomial and $\deg f> \deg c_k$. Suppose that $f$  has degree $2$ in the variable $z$ and it is multilinear in the variables $y$'s, then $f$ is a linear combination of polynomials 
\[
[z, y_{i_1},\ldots, y_{i_t}, z, y_{i_t+1},\ldots y_{i_l}],
\]
where the rightmost block of variables $Y$ is not empty.
\end{lem}
\begin{proof}
Note that the  multihomogeneous consequences of the polynomial $c_{k} \pmod{I}$,  that have degree $2$ in the variable $z$ and degree $1$ in all variables $y$, are obtained by applying a combination of the following rules:
\begin{itemize}
    \item Making substitutions of type $z\mapsto z+ [z,y_{i_1},\ldots, y_{i_t}]$.
    \item Making $[c_k,y_{i_1},\ldots, y_{i_t}]$.
\end{itemize}
Now replacing $z$ by $z+[z,y_{k+1},\ldots, y_{k+n}]$ in $c_{k}$ and taking the homogeneous component of degree $1$ in $y_{k+1}$, \dots, $y_{k+n}$, we obtain the following polynomial 
\[
g= [ z,y_1,\ldots,y_k,y_{k+1},\ldots,y_{k+n}, z]+[z,y_1,\ldots, y_k, [z, y_{k+1}, \ldots, y_{k+n}]]\pmod{I},
\]
where we apply to the first summand several times the identity (\ref{zcomutanUT3}). 

Let $C=\{k+1,\ldots, k+n\}$ and let $P(C)=\{S\mid S\subseteq C\}$ be the set of all subsets of $C$.  If $S=\{i_1,\ldots,i_t\}\subseteq C$, and its complement in $C$ is $S^c=C\setminus S=\{i_{t+1},\ldots, i_n\}$, we suppose that $i_1<\cdots < i_t$ and $i_{t+1}<\cdots < i_n$. Define
\[
[z,y_1,\ldots, y_k, Y_S, z, Y_{S^c}]= [z, y_1,\ldots, y_k, y_{i_1},\ldots, y_{i_t}, z, y_{i_{t+1}},\ldots, y_{i_n}].
\]
Recall that $\ch F=2$, so using several times the fact that $ad\, y=[* , y]$ is a derivation in every Lie algebra, we have 
\[
[z,y_1,\ldots, y_k, [z, y_{k+1}, \ldots, y_{k+n}]]=\sum\limits_{S\in P(C)} [z,y_1,\ldots, y_k, Y_S, z, Y_{S^c}],
\]
therefore 
\begin{align*}
g &= [ z,y_1,\ldots,y_k,y_{k+1},\ldots,y_{k+n}, z]+\sum\limits_{S\in P(C)} [z,y_1,\ldots, y_k, Y_S, z, Y_{S^c}]\\
&=  2 [ z,y_1,\ldots,y_k,y_{k+1},\ldots,y_{k+n}, z]+\sum\limits_{S\in P(C), S\neq C} [z,y_1,\ldots, y_k, Y_S, z, Y_{S^c}]\\
& =  \sum\limits_{S\in P(C), S\neq C} [z,y_1,\ldots, y_k, Y_S, z, Y_{S^c}],
\end{align*}
and the result follows.
\end{proof}

\begin{theo} 
If $F$ is an infinite field, $\ch F=2$, then $I$, the ideal of the $\mathbb{Z}_3$-graded identities of $UT_3^{(-)}$ does not have the Specht property.
\end{theo}

\begin{proof}
Given $k\geq 1$, as above, we consider the polynomials
\[
c_k(z,y_1,\ldots,y_k)= [z,y_1,\ldots, y_k,z].
\]
Note that by Theorem \ref{idutnf}, $c_k\neq 0 \pmod{I}$. 
We perform the following substitutions in the polynomial $c_k$ 
\begin{itemize}
\item $z= e_{12}+ e_{23}$,
\item $y_i= \gamma_i ^{1}e_{11}+ \gamma_i ^{2}e_{22}+ \gamma_i ^{3}e_{33}$,
\end{itemize}
where $\gamma_i^j$ are commuting independent variables. In this way we obtain the following expression
\begin{equation}\label{ck}
 \left(  \prod\limits_{i=1}^k (\gamma_i^1+ \gamma_i^2)+   \prod\limits_{i=1}^k (\gamma_i^2+ \gamma_i^3)\right) e_{13}
\end{equation}
Define
\[
h= \left(  \prod\limits_{i=1}^k (\gamma_i^1+ \gamma_i^2)+   \prod\limits_{i=1}^k (\gamma_i^2+ \gamma_i^3)\right), 
\]
a polynomial in the  commuting variables $\gamma_i^l$ where $1\leq i\leq k$ and $1\leq l \leq 3$. 

Let $I_k=\{1,\ldots, k\}$ and consider $S\subseteq I_k$ such that $S\neq \emptyset$ and $S\neq I_k$. 
We define the following polynomial
\[
f_S=[z, y_{i_1},\ldots, y_{i_t}, z, y_{i_{t+1}},\ldots, y_{i_k}],
\]
where $S=\{i_1,\ldots, i_t\}$, $S^c=\{i_{t+1},\ldots, i_k\}$, $i_1<\cdots< i_t$ and $i_{t+1}<\cdots<i_k$. By Theorem \ref{idutnf}, $f_S\neq 0\pmod {I}$. Making the above substitution in $f_S$, we get
\begin{equation}\label{cks}
 \left(  \prod\limits_{i\in S} (\gamma_i^1+ \gamma_i^2)+   \prod\limits_{i\in S} (\gamma_i^2+ \gamma_i^3)\right) \prod\limits_{j\in S^c} (\gamma_j^1+ \gamma_j^3)e_{13}.
\end{equation}
Consider 
\[
h_S= \left(  \prod\limits_{i\in S} (\gamma_i^1+ \gamma_i^2)+   \prod\limits_{i\in S} (\gamma_i^2+ \gamma_i^3)\right) \prod\limits_{j\in S^c} (\gamma_j^1+ \gamma_j^3).
\]
Observe that $h_S$ is a polynomial in the  commuting variables $\gamma_i^l$ where $1\leq i\leq k$ and $1\leq l \leq 3$. Notice that
\[
h\notin \spn_{F}\{h_S\mid S\subseteq I_k, S\neq I_k, S\neq\emptyset\},
\]
because the monomials $\gamma_1^{1}\prod_{i=2}^k \gamma_i^{2}$ and  $\gamma_1^{3}\prod_{i=2}^k \gamma_i^{2}$  appear in the polynomial $h$, but these monomials do not appear in any of the $h_S$.

Let $J_k=\langle c_1,\ldots c_k\rangle_{T_{\mathbb{Z}_3}} \pmod{I}$ and suppose $c_k\in J_{k-1}$. By Lemma \ref{conseck}, we get
\[
c_k=\sum\limits_{S\in\mathcal{P}} \alpha_{S} f_S
\]
where $\alpha_S\in F$ and $S\in\mathcal{P}=\{S\subseteq I_k\mid S\neq I_k, S\neq\emptyset\}$. Then (\ref{ck}) and (\ref{cks}) imply 
\[
h=\sum\limits_{S\in\mathcal{P}} \alpha_S h_S,
\]
but $h\notin \spn_{F}\{h_S\mid S\in\mathcal{P}\}$.
Therefore $c_k\notin J_{k-1}$ and we have that the following ascending chain of $T_{\mathbb{Z}_3}$-ideals modulo $I$
\[
J_1\subset J_2\subset\cdots\subset J_k\subset\cdots,
\]
where $J_i=\langle c_1,\ldots c_i\rangle_{T_{\mathbb{Z}_3}}$, is not stationary  (does not stabilize). Therefore the  $T_{\mathbb{Z}_3}$-ideal $I$ does not have the Specht property.
\end{proof}

\section{The graded identities of $T_{\mathbb{Z}_n}(UT_n(F))$}

Let $Z^{g_i}$ and $Y$ be disjoint infinite countable sets, $Z^{g_i}= \{z_1^{g_i},z_2^{g_i}\ldots,\}$ where $g_i\in\mathbb{Z}_n\setminus \{0\}$ and $Y=\{y_1,y_2,\ldots\}$. Consider the free Lie algebra
\[
\mathcal{L}_{\mathbb{Z}_n}=\mathcal{L}\bigg(\bigcup_{g_i\in\mathbb{Z}_n\setminus \{0\}} Z^{g_i}\cup Y\bigg)
\]
freely generated over $F$ by $\bigcup_{g_i\in\mathbb{Z}_n\setminus \{0\}} Z^{g_i}\cup Y$. Notice that the algebra $\mathcal{L}_{\mathbb{Z}_n}$ has a natural $\mathbb{Z}_n$-grading assuming the variables $Y$, $Z^{g_i}$ to be of homogeneous degrees $0$ and $g_i$, for $g_i\in\mathbb{Z}_n\setminus\{0\}$, respectively. Consider the algebra $UT_n(F)$  endowed with the canonical $\mathbb{Z}_n$-grading, and denote by $I$ the ideal  of $\mathbb{Z}_{n}$-graded identities for $UT_n(F)$.

By Theorem \ref{idutnf}, we get that the non-zero monomials in the algebra  $\mathcal{L}_{\mathbb{Z}_n}/I$ are of the following types
\begin{enumerate}[label=\Roman*.]
    \item $[z^{g_i},y_{i_1},\ldots,y_{i_s}]$,
    \item $[z^{g_{j_1}},y^{(1)}_{i_1},\ldots,y^{(1)}_{i_s},z^{g_{j_2}}, y^{(2)}_{l_1},\ldots,y^{(2)}_{l_m},\ldots, z^{g_{j_k}}, y^{(k)}_{t_1},\ldots,y^{(
    k)}_{t_n}]$,
\end{enumerate}
where $\sum\limits_{j=1}^k g_{i_j}\leq n-1$ and the indices of the variables $y$'s are ordered in non-decreasing way (that is their indices, in each group, increase with possible repetitions). Observe that some (or all) of the blocks of variables $y$'s can be empty in monomials of type II.

Let $z_1$, \dots, $z_{k}$ be variables of degree $g_1$, \dots, $g_k$, respectively, such that $\sum\limits_{i}^{k} g_i\leq n-1$ and $g_i\neq 0$. Denote by
\begin{align*}
[z_i,a_1y_1,&\ldots,a_ny_n]= [z,\underbrace{y_1,\ldots,y_1}_{a_1},\ldots,\underbrace{y_n,\ldots,y_n}_{a_n}]\\
[z_1,a_1^{(1)}y_1,&\ldots,a_n^{(1)}y_n,\ldots, z_k, a_1^{(k)}y_1,\ldots,a_n^{(k)}y_n]\\
&=[z_1,\underbrace{y_1,\ldots,y_1}_{a_1^{(1)}},\ldots,\underbrace{y_n,\ldots,y_n}_{a_n^ {(1)}}, \ldots, z_k, \underbrace{y_1,\ldots,y_1}_{a_1^{(k)}},\ldots, \underbrace{y_n,\ldots, y_n}_{{a_n}^{(k)}}].
\end{align*}
Define the following sets
\begin{align*}
&\mathcal{B}_{g_i}=\{[z_{i},a_1^{(1)}y_1,\ldots,a_n^{(1)}y_n]\mid a_i>0,\ n\geq 1\};\\
&\mathcal{B}_{(g_1,\ldots, g_k)}=\{[z_1,a_1^{(1)}y_1,\ldots,a_n^{(1)} y_n, z_{\sigma(2)},a_1^{(2)}y_1,\ldots,a_n^{(2)}y_n,\ldots, z_{\sigma(k)},a_1^{(k)}y_1,\ldots,a_n^{(k)}y_n]\},
\end{align*}
where $\sigma\in S_k$, $\sigma(1)=1$, $a^{(j)}_i\geq 0$, $1\leq j\leq n$, and $n\geq 0$. 

\begin{defi}
Let $f=[z_{i},a_1y_1,\ldots,a_ny_n]$, $g =[z_{i},a'_1y_1, \ldots, a'_my_m]$ be polynomials in $\mathcal{B}_{g_i}$. Consider the finite sequences in $D(\mathbb{N})$
\[
V_f=(a_1,\ldots, a_n), \quad V_g= (a'_1,\ldots, a'_m).
\]
We define the following order on $\mathcal{B}_{g_i}$: $f\leq_{\mathcal{B}_{g_i}} g$ whenever $V_f\leq V_g$.
\end{defi}

\begin{prop}\label{Bgiutn}
Consider the set $\mathcal{B}_{g_i}$ and let $f$, $g\in  \mathcal{B}_{g_i}$
be two polynomials in $\mathcal{B}_i$.
\begin{enumerate}[label=\roman*.]
\item 
If $f\leq_{\mathcal{B}_{g_i}} g$, then $g$ is a consequence of $f$ modulo $I$.
\item 
$(\mathcal{B}_{g_i}, \leq_{\mathcal{B}_{g_i}})$ satisfies the f.b.p.
\item 
Let $J$ be a $T_{\mathbb{Z}_n}$-ideal such that $I\subseteq J$. Consider the following sets 
\[
\mathcal{A}_{g_i}=\{f\in J\mid f\in \mathcal{B}_{g_i}\}.
\]
Then there exist finite subsets $\mathcal{A}'_{g_i}\subseteq \mathcal{A}_{g_i}$  such that 
\[
\langle \mathcal{A}_{g_i}\rangle_{T_{\mathbb{Z}_n}}= \langle \mathcal{A}'_{g_i}\rangle_{T_{\mathbb{Z}_n}} \pmod{I}.
\]
\end{enumerate}
\end{prop}
\begin{proof}
The proof of this proposition is completely analogous to the proofs of Proposition \ref{fbporUT3} and Corollary \ref{Bifinito}, and that is why we omit it.
\end{proof}
\begin{defi}\label{ordeg1gk}
Let $f$, $g\in\mathcal{B}_{(g_1,\ldots, g_k)}$ be respectively the polynomials
\begin{align*}
[z_1,a_1^{(1)}y_1,\ldots,a_n^{(1)}y_n, z_{\sigma(2)},a_1^{(2)}y_1,\ldots,a_n^{(2)}y_n,\ldots, z_{\sigma(k)}, ,a_1^{(k)}y_1,\ldots,a_n^{(k)}y_n],\\
[z_1,{a'_1}^{(1)}y_1,\ldots,{a'_m}^{(1)}y_m, z_{\sigma(2)},{a'_1}^{(2)}y_1,\ldots,{a'_m}^{(2)}y_m,\ldots, z_{\sigma(k)}, {a'_1}^{(k)}y_1,\ldots,{a'_m}^{(k)}y_m]
    \end{align*}
where $\sigma$, $\tau\in S_k$ are such that $1=\sigma(1)=\tau(1)$. Consider the finite sequences
\begin{align*}
V_f&=((a^{(1)}_1,\ldots,a^{(k)}_1),\ldots, (a^{(1)}_n,\ldots,a^{(k)}_n)), \\ V_g&=(({a'_1}^{(1)},\ldots,{a'_1}^{(k)}),\ldots, ({a'_m}^{(1)},\ldots,{a'_m}^{(k)})).
\end{align*}
If $V_f\leq_k V_g$ and $\sigma=\tau$, we write $f\leq_{\mathcal{B}_{(g_1,\ldots, g_k)}} g$.
\end{defi}
Recall that  $V_f$, $V_g\in D({\mathbb{N}_0}^k)$. The order $\leq_k$ coincides with the order defined in Example \ref{odernk}.
\begin{lem}
Let $f\in\mathcal{B}_{g_1,\ldots,g_k}$ be such that
\[
f(z_1,\ldots,z_k,y_1,\ldots,y_k)=[z_1,y_1,z_2,y_2,\ldots,z_k,y_k].
\]
Then the polynomial
\[
g=[z_1,y_1,\ldots,z_{i-1},y_{i-1}, z_i, y, y_i,z_{i+1},\ldots, z_k,y_k]\in\langle f\rangle_{T_{\mathbb{Z}_n}} \pmod{I}.
\]
\end{lem}
\begin{proof}
If $i=1$ then $g= f([z_1,y],z_2,\ldots, z_k,y_1,\ldots, y_k) \in\langle f\rangle_{T_{\mathbb{Z}_n}}\pmod{I}$. If $i=k$ then $g=[f,y]$ and $g$ is a consequence of $f$.

Thus we suppose $1<i<k$. Let $f_s$ be the polynomial obtained replacing the variable $z_s$ by $[z_s,y]$ in $f$, where $1\leq s\leq i$. Notice that $f_s$ equals 
\[
[z_1,y_1,\ldots,z_s,y,y_s,z_{s+1},y_{s+1}\ldots,z_k, y_k]- [z_1,y_1,\ldots, z_{s-1}, y_{s-1}, y, z_s,y_s,\ldots, z_k,y_k],
\]
where $s\in\{1,\ldots, i\}$.
Recall that $[y,y']=0$ mod $I$, thus one has
\[
g=\sum\limits_{s=1}^i f_s \pmod{I}.
\]
Therefore $g$ is a consequence of $f$ modulo $I$.
\end{proof}

The previous lemma has as a direct consequence the following result. 
\begin{prop}\label{conseutn}
Let $f$, $g\in \mathcal{B}_{(g_1,\ldots, g_k)}$ be polynomials such that $V_f\leq_{\mathcal{B}_{(g_1,\ldots, g_k)}} V_g$. Then $g\in\langle f\rangle_{T_{\mathbb{Z}_n}}$ modulo $I$.
\end{prop}

\begin{prop}
The set $(\mathcal{B}_{(g_1,\ldots, g_k)}, \leq_{\mathcal{B}_{(g_1,\ldots, g_k)}})$ satisfies the f.b.p.
\end{prop}
\begin{proof}
Suppose that there is an infinite sequence $\{f_i\}_{i\geq 1}$ of  pairwise incomparable elements in $\mathcal{B}_{(g_1,\ldots, g_k)}$ with respect to the order $\leq_{\mathcal{B}_{(g_1,\ldots, g_k)}}$.  Since the symmetric group $S_k$  is a finite set, we obtain an infinite subsequence $\{f_{i_j}\}_{j\geq i}$
from $\{f_i\}_{i\geq 1}$, such that 
the permutations $\sigma_{i_j}$ are the same for all $j\geq 1$. Hence, the previous subsequence induces the sequence
$\{V_{f_{i_j}}\}_{j\geq i}$ in  $D({\mathbb{N}}^k)$. 
By Definition \ref{ordeg1gk}, these elements are pairwise incomparable with respect to the order $\leq_k$. But this is absurd because $(D({\mathbb{N}_0}^k),\leq_k)$ is partially well-ordered.
\end{proof}

\begin{prop}\label{liutn}
Let $n\geq 3$ and take $z_1$, \dots, $z_k$ variables of degrees $g_1$, \dots, $g_k$, respectively, such that $g_i\ne 0$ and $\sum\limits_{i=1}^k g_i\leq n-1$. Let $y_i$ be variables of degree 0. Then the commutators 
\[
c=[z_1, y_1, \ldots, y_{t_1}, z_2, y_{t_1+1}, \ldots, y_{t_1+t_2}, z_3, \ldots, z_k, y_{t_1+\cdots+t_{k-1}+1}, \ldots, t_{t_1+\cdots+t_k}]
\]
are linearly independent modulo the graded identities for $UT_n(F)^{(-)}$.
\end{prop}
\begin{proof}
Suppose that $m-1=g_1+\cdots+g_k$ and consider first the substitutions: 
\[
z_1=e_{1,g_1+1}, \quad z_2=e_{g_1+1, g_1+g_2+1},  \ldots, \quad z_k=e_{g_1+\cdots+g_{k-1}+1, g_1+\cdots+g_k+1}.
\]
Pay attention $g_1+\cdots+g_k+1=m\le n$.

A standard staircase argument shows that, fixing $z_1$ at the leftmost position in the commutator, the only permutation of the $z_2$, \dots, $z_k$ that yields a nonzero element will be as in the commutator $c$ given in the statement of the proposition.

Computing $c$ with the above substitution of the $z_i$, and assuming $y_i= y_i^1e_{11}+y_i^2 e_{22}+\cdots+y_i^ne_{nn}$ where the $y_i^j$ are commuting independent variables, yields
\[
c=\prod_{s=1}^k \prod_i (y_i^{g_1+\cdots+g_s+1} - y_i^1) e_{1m}.
\]
Here in the second product, $i$ ranges:
\begin{itemize}
\item From $1$ to $t_1$, if $s=1$.
\item From $t_1+\cdots+t_{s-1}+1$ to $t_1+\cdots+t_s$, if $s\geq 2$.
\end{itemize}
Now make the following substitutions for the $y_i$.

First put all of the $y_i^j=0$ \textbf{except for}:

\begin{itemize}
\item
$i=1$, \dots, $t_1$: here define $y_i^{g_1+1}=1$.
\item
$i=t_1+1$, \dots, $t_1+t_2$: here define $y_i^{g_1+g_2+1}=1$.
\item
And so on, for $i=t_1+\cdots+t_{s-1}+1$ to $i=t_1+\cdots+t_s$ define $y_i^{g_1+\cdots+g_s+1}=1$.
\end{itemize}
Such a substitution vanishes all commutators but $c$, and $c$ evaluates to $c=e_{1m}$. 

Let us suppose there is a nontrivial linear combination among multihomogeneous commutators of the type $c$, such that $c$ participates with a nonzero coefficient in it. 

First we get, by means of the above substitution for the $z_i$, the correct order for the $z_2$, \dots, $z_k$. Then we make the substitutions for the $y_i$ as above. All commutators vanish except for $c$. This implies the coefficient of $c$ must be 0, a contradiction. This proves the linear independence.
\end{proof}

\begin{prop}\label{liutnh}
Let $n\geq 3$ and let $z_1$, \dots, $z_k$ be variables of degrees $g_1$, \dots, $g_k$, respectively, such that $g_i\ne 0$ and $\sum_{i=1}^k g_i\leq n-1$. Take $y_i$ variables of degree 0. Then the commutators 
\[
c=[z_1,a_1^{(1)}y_1,\ldots,a_t^{(1)}y_t, z_{2},a_1^{(2)}y_1,\ldots,a_t^{(2)}y_t,\ldots, z_{k}, a_1^{(k)}y_1,\ldots,a_t^{(k)}y_t]
\]
are linearly independent modulo the graded identities of $UT_n(F)^{(-)}$.
\end{prop}
\begin{proof}
Suppose that $m-1=g_1+\cdots+g_k$ and consider first the substitution: 
\[
z_1=e_{1,g_1+1}, \quad
z_2=e_{g_1+1, g_1+g_2+1}, \quad\ldots, \quad
z_k=e_{g_1+\cdots+g_{k-1}+1, g_1+\cdots+g_k+1}.
\]
(Pay attention that $g_1+\cdots+g_k+1=m$.)

Once again, a staircase argument shows that, fixing $z_1$ at the leftmost position in the commutator, the only permutation of the $z_2$, \dots, $z_k$ that yields a nonzero element will be as in the commutator $c$.

Computing $c$ with the above substitution of the $z_i$, and assuming $y_i= y_i^1e_{11}+y_i^2 e_{22}+\cdots+y_i^ne_{nn}$, where the $y_i^j$ are commuting independent variables, gives
\[
c=\prod_{s=1}^k \prod_{i=1}^t (y_i^{g_1+\cdots+g_s+1} - y_i^1)^{a_i^{(s)}} e_{1m}.
\]
Now making the substitution $y_i^1=0$ for all $i$, we obtain:
\begin{equation}\label{cutn}
c= \prod_{s=1}^k \prod_{i=1}^t (y_i^{g_1+\cdots+g_s+1} )^{a_i^{(s)}} e_{1m}.
\end{equation}
Define the following monomial
\[
m_c= \prod_{s=1}^k \prod_{i=1}^t (y_i^{g_1+\cdots+g_s+1} )^{a_i^{(s)}}.
\]
Note that if $c$ and $c'$ are different commutators with the same permutation of the $z_2$, \dots, $z_k$, we have that the monomials $m_{c}$ and $m_{c'}$ are different. We suppose there is a nontrivial linear combination among commutators of type $c$
\[
\sum\limits_{i=1}^r \alpha_i c_i(z_1,\ldots, z_k,y_1,\ldots, y_t)
\]
which is a  graded identity for $UT_n(F)^{(-)}$.
First we get, by means of the substitution for the $z_i$, the correct order for the $z_2$, \dots, $z_k$ with respect to the commutator $c_1$. Then we make the substitution for the $y_i$ as above. By Equality (\ref{cutn}), we have
\[
0=\Bigl(\sum\limits_{j=1}^l \alpha_{i_j} m_{c_{i_j}}\Bigr)e_{1m},
\]
where the polynomials $c_{i_j}$ have the same permutation of the $z_2$, \dots, $z_k$ as the polynomial $c_1$. It follows that $0=\sum\limits_{j=1}^l \alpha_{i_j} m_{c_{i_j}}$.

But the monomials $m_{c_{i_j}}$ (in commuting variables) are  linearly independent because they are different for each $c_{i_j}$, hence $\alpha_{i_j}=0$. Using the same argument several times, we have that $\alpha_i=0$ for every $1\leq i\leq r$, and the claim follows.
\end{proof}

\begin{defi}\label{linordutn}
Let $f_i$, $f_j\in\mathcal{B}_{(g_1,\ldots, g_k)}$ be multihomogeneous  polynomials of the same multidegree, as given below, respectively
\begin{align*}
&[z_1,a_1^{(1)}y_1,\ldots,a_n^{(1)}y_n, z_{\sigma(2)},a_1^{(2)}y_1,\ldots,a_n^{(2)}y_n,\ldots, z_{\sigma(k)} ,a_1^{(k)}y_1,\ldots,a_n^{(k)}y_n];\\
    & [z_1,{a'_1}^{(1)}y_1,\ldots,{a'_n}^{(1)}y_n, z_{\tau(2)},{a'_1}^{(2)}y_1,\ldots,{a'_n}^{(2)}y_n,\ldots, z_{\tau(k)}, {a'_1}^{(k)}y_1,\ldots,{a'_n}^{(k)}y_n].
\end{align*}
Consider the finite sequences
\begin{align*}
V_{f_i}&=((a^{(1)}_1,\ldots,a^{(k)}_1),\ldots, (a^{(1)}_n,\ldots,a^{(k)}_n)), \\V_{f_j}&=(({a'_1}^{(1)},\ldots,{a'_1}^{(k)}),\ldots, ({a'_n}^{(1)},\ldots,{a'_n}^{(k)})).
\end{align*}
We define the order
$f_j\leq'  f_i$ if some of the following conditions is met:
\begin{itemize}
\item 
$(\sigma(2),\ldots,\sigma(k))<_{lex} (\tau(2),\ldots,\tau(k))$;
\item 
$\sigma=\tau$ and
$\sum_{s=1}^{k-1}\sum_{i=1}^n {a'_i}^{(s)}> \sum_{s=1}^{k-1}\sum_{i=1}^n {a_i}^{(s)} $;
\item $\sigma=\tau$, $\sum_{s=1}^{k-1}\sum_{i=1}^n {a'_i}^{(s)}=\sum_{s=1}^{k-1}\sum_{i=1}^n {a_i}^{(s)} $ and $(\sum_{i=1}^n{a_i}^{(1)},\ldots, \sum_{i=1}^n{a_i}^{(k-1)})<_{\textrm{lex}} (\sum_{i=1}^n{a'_i}^{(1)},\ldots, \sum_{i=1}^n{a'_i}^{(k-1)})$;
\item $\sigma=\tau$, $\sum_{s=1}^{k-1}\sum_{i=1}^n {a'_i}^{(s)}=\sum_{s=1}^{k-1}\sum_{i=1}^n {a_i}^{(s)} $,  $(\sum_{i=1}^n{a_i}^{(1)},\ldots, \sum_{i=1}^n{a_i}^{(k-1)})=(\sum_{i=1}^n{a'_i}^{(1)},\ldots, \sum_{i=1}^n{a'_i}^{(k-1)})$ and $({a_1}^{(1)},\ldots,{a_n}^{(1)}, \ldots, {a_1}^{(k-1)}, \ldots, {a_n}^{(k-1)})\leq_{lex} ({a'_1}^{(1)},\ldots,{a'_n}^{(1)}, \ldots, {a'_1}^{(k-1)}, \ldots, {a'_n}^{(k-1)})$.
\end{itemize}
\end{defi}
Observe that this order is linear on the polynomials of the same multidegree in $\mathcal{B}_{(g_1,\ldots,g_k)}$.

Suppose $f=\sum_{i=1}^t f_i$ is a multihomogeneous polynomial, where $f_i\in\mathcal{B}_{(g_1,\ldots, g_k)}$ and $\alpha_i\in F\setminus\{0\}$. By Propositions \ref{liutn} and \ref{liutnh},   this way of expressing $f$ as a linear combination of elements of $\mathcal{B}_{(g_1,\ldots, g_k)}$ is unique. For this reason, we can define the leading monomial of $f$ with respect to the order $\leq'$.

\begin{defi}\label{mlUTn}
Let $f=\sum_{i=1}^t \alpha_i f_i$ be a multihomogeneous polynomial, where $f_i\in\mathcal{B}_{(g_1,\ldots, g_k)}$ and $\alpha_i\in F\setminus \{0\}$. We define the leading monomial of $f$ by
\[
\textrm{ml}(f)= \max_{\leq'} \{f_i\mid 1\leq i\leq n\}, 
\]
and the leading coefficient of $f$, $\textrm{cl}(f)$, as the coefficient of $\textrm{ml}(f)$. 
\end{defi}
\begin{prop}\label{oderpreserutn}
Consider two elements in $\mathcal{B}_{(g_1,\ldots, g_k)}$
\begin{align*}
&f_i=[z_1,a_1^{(1)}y_1,\ldots,a_n^{(1)}y_n, z_{\sigma(2)},a_1^{(2)}y_1,\ldots,a_n^{(2)}y_n,\ldots, z_{\sigma(k)}, a_1^{(k)}y_1,\ldots,a_n^{(k)}y_n];\\
&f_j=[z_1,{a'_1}^{(1)}y_1,\ldots,{a'_n}^{(1)}y_n, z_{\tau(2)},{a'_1}^{(2)}y_1,\ldots,{a'_n}^{(2)}y_n,\ldots, z_{\tau(k)}, {a'_1}^{(k)}y_1,\ldots,{a'_n}^{(k)}y_n]
\end{align*}
of the same multidegree and suppose that $f_j\leq ' f_i$. Let $f_i^{(s)}$, $f_j^{(s)}$ be the polynomials obtained by replacing the variable $z_s$ by $[z_s, y]$ in $f_i$ and $f_j$, respectively, where $1\leq s\leq k$. Then $ml(f_j^{(s)})\leq' ml(f_i^{(s)})$.
\end{prop}
\begin{proof}
Consider the finite sequences
\begin{align*}
V_{f_i}&=((a^{(1)}_1,\ldots,a^{(k)}_1),\ldots, (a^{(1)}_n,\ldots,a^{(k)}_n)),\\
V_{f_j}&=(({a'_1}^{(1)},\ldots,{a'_1}^{(k)}),\ldots, ({a'_n}^{(1)},\ldots,{a'_n}^{(k)})).
\end{align*}
If $(\sigma(2),\ldots,\sigma(k))<_{lex}(\tau(2),\ldots,\tau(k))$, the result follows.

Now consider $\sigma=\tau$ and without loss of generality suppose that $\sigma$ is the identity. Hence
\begin{align*}
f_i^{(s)} & =  [z_1,{a_1}^{(1)}y_1,\ldots,{a_n}^{(1)}y_n,\ldots,z_s,y,{a_1}^{(s)}y_1,\ldots, {a_n}^{(s)}y_n,z_{s+1}, \\ 
&\qquad \qquad \qquad \qquad {a_1}^{(s+1)}y_1,\ldots,{a_n}^{(s+1)}y_n\ldots,z_k, {a_1}^{(k)}y_1,\ldots, {a_n}^{(k)}y_n]\\
&  - [z_1,{a_1}^{(1)}y_1,\ldots, {a_n}^{(1)}y_n\ldots, z_{s-1}, {a_1}^{(s-1)}y_1,\ldots, {a_n}^{(s-1)}y_n, y, z_s,\\
&\qquad \qquad \qquad \qquad {a_1}^{(s)}y_1,\ldots, {a_n}^{(s)}y_n,\ldots, z_k,{a_1}^{(k)}y_1,\ldots, {a_n}^{(k)}y_n];
\end{align*}

\begin{align*}
f_j^{(s)} & =  [z_1,{a'_1}^{(1)}y_1,\ldots,{a'_n}^{(1)}y_n,\ldots,z_s,y,{a'_1}^{(s)}y_1,\ldots,{a'_n}^{(s)}y_n,z_{s+1},\\
&\qquad \qquad \qquad \qquad {a'_1}^{(s+1)}y_1,\ldots,{a'_n}^{(s+1)}y_n,\ldots,z_k, {a'_1}^{(k)}y_1,\ldots, {a'_n}^{(k)}y_n]\\
& - [z_1,{a'_1}^{(1)}y_1,\ldots, {a'_n}^{(1)}y_n\ldots, z_{s-1}, {a'_1}^{(s-1)}y_1,\ldots, {a'_n}^{(s-1)}y_n, y, z_s,\\
&\qquad \qquad \qquad \qquad {a'_1}^{(s)}y_1,\ldots, {a'_n}^{(s)}y_n,\ldots, z_k,{a'_1}^{(k)}y_1,\ldots, {a'_n}^{(k)}y_n].
\end{align*}
By Definition \ref{linordutn}, we have
\begin{align*}
ml(f_i^{(s)})&=[z_1,{a_1}^{(1)}y_1,\ldots,{a_n}^{(1)}y_n,\ldots,z_s,y,{a_1}^{(s)}y_1,\ldots, {a_n}^{(s)}y_n,z_{s+1},\\
&\qquad \qquad \qquad \qquad {a_1}^{(1)}y_1,\ldots,{a_n}^{(s+1)}y_n\ldots,z_k, {a_1}^{(k)}y_1,\ldots, {a_n}^{(k)}y_n];\\ 
     ml(f_j^{(s)})&= [z_1,{a'_1}^{(1)}y_1,\ldots,{a'_n}^{(1)}y_n,\ldots,z_s,y,{a'_1}^{(s)}y_1,\ldots,{a'_n}^{(s)}y_n,z_{s+1},\\
&\qquad \qquad \qquad \qquad {a'_1}^{(1)}y_1,\ldots,{a'_n}^{(s+1)}y_n\ldots,z_k, {a'_1}^{(k)}y_1,\ldots, {a'_n}^{(k)}y_n].
\end{align*}
Observe  that the $s$-th block of variables $y$'s in both polynomials  $cl(f_i^{(s)})$ and $cl(f_j^{(s)})$ is subjected to the same modification by the same variable $y$. Recall that
\[
[z,y_1,y_2]=[z,y_2,y_1]\pmod{I}
\]
for every variable $z$ of degree different from $0$.  Hence the order is preserved, and this means  $ml(f_j^{(s)})\leq' ml(f_i^{(s)})$.
\end{proof}
\begin{cor}\label{oderpreserutncomblinear}
Consider  $f=\sum_{i=1}^t \alpha_i f_i$
a multihomogeneous polynomial, where $f_i\in\mathcal{B}_{(g_1,\ldots, g_k)}$, and suppose that $ml(f)=f_1$. Let $f_i^{(s)}$ be the polynomial obtained by replacing the variable $z_s$ by $[z_s, y]$ in $f_i$, where $1\leq s\leq k$ and $1\leq i\leq t$. Then $ml(f_i^{(s)})\leq' ml(f_1^{(s)})$.
\end{cor}
\begin{proof}
Recall that the commutators $f_i$ have the same  multidegree. So, applying the previous proposition, the result follows.
\end{proof}
\begin{prop}\label{seq2orderutn}
Consider $f$, $g$ two multihomogeneous polynomials such that
\[
f=\sum\limits_{i=1}^t \alpha_i f_i, \qquad g=\sum\limits_{j=1}^s \beta_j g_j,
\]
where $\alpha_i$, $\beta_j\in F\setminus\{0\}$, $f_i,g_j\in\mathcal{B}_{(g_1,\ldots, g_k)}$ for every $1\leq i\leq t$ and $1\leq j\leq s$. Suppose that $ml(f)\leq_{\mathcal{B}_{(g_1,\ldots, g_k)}} ml(g)$. Then there exists $h\in\langle f\rangle_{\mathbb{Z}_n}$ modulo $I$ such that $ml(h)=ml(g)$ and $cl(h)=cl(f)$. 
\end{prop}
\begin{proof}
By Proposition \ref{conseutn}, we have that $ml(g)$ is a consequence of $ml(f)$ modulo $I$. Then Corollary \ref{oderpreserutncomblinear}, making the same computations as in the case of $ml(f)$ in order to obtain $ml(g)$, we deduce a consequence $h$ from $f$ such that $ml(h)=ml(g)$. Moreover, the leading coefficient of $h$ is the same as that of $f$.
\end{proof}
\begin{defi}
Let $f$ be a multihomogeneous polynomial that is a linear combination of polynomials in $\mathcal{B}_{(g_1,\ldots, g_k)}$. Then $f$ is called a polynomial of type $(g_1,\ldots,g_k)$.
\end{defi}
\begin{prop}\label{sequtn}
There is no infinite sequence of polynomials  $\{f_i\}_{i\geq 1}$ of type $(g_1,\ldots,g_k)$
such that
\[
f_i\notin \langle f_1,\ldots, f_{i-1}\rangle_{T_{\mathbb{Z}_n}} \pmod{I}
\]
for every $i\geq 2$.
\end{prop}
\begin{proof}
The proof is completely analogous to that of Proposition \ref{seq2orderut3}.
\end{proof}
As a consequence of the previous result, we have the following corollary 
\begin{cor}\label{Bg1gkfinito}
Let $J$ be a $T_{\mathbb{Z}_n}$-ideal such that $I\subseteq J$. Consider the following set 
\[
\mathcal{A}_{(g_1,\ldots,g_k)}=\{f\in J\mid f\ \textrm{is a polynomial of type $(g_1,\ldots, g_k)$}\}.
\]
Then there exists a finite subset $\mathcal{A}' _{(g_1,\ldots,g_k)}\subseteq \mathcal{A}_{(g_1,\ldots,g_k)}$ such that 
\[
\langle\mathcal{A}_{(g_1,\ldots,g_k)}\rangle_{T_{\mathbb{Z}_n}}= \langle \mathcal{A}'_{(g_1,\ldots, g_k)}\rangle_{T_{\mathbb{Z}_n}} \pmod{I}.
\]
\end{cor}

\begin{theo}
Suppose that $\ch F=0$ or $\ch F\geq n$. If $J$ is a $T_{\mathbb{Z}_n}$-ideal such that $I\subseteq J$, then $J$ is finitely generated as a $T_{\mathbb{Z}_n}$-ideal.
\end{theo}
\begin{proof}
Since $F$ is an infinite field, $J$ is generated by its multihomogeneous polynomials. If $\ch F\geq n$ or $\ch F=0$, using the multilinearization process, we can consider that any multihomogeneous polynomial is linear in the variables of degree different from $0$, because by Theorem \ref{idutnf}, each of them can appear in the non-zero monomials of  ${\mathcal{L}_{\mathbb{Z}_n}}/I$ at most $n-1$ times. Hence, $J$ is generated as $T_{\mathbb{Z}_n}$-ideal, modulo $I$, by the following sets
\begin{itemize}
    \item  $\mathcal{A}_{g_i}=\{f\in J\mid f\in \mathcal{B}_{g_i}\}$ where $g_i\in\mathbb{Z}_n\setminus\{0\}$;
    \item $\mathcal{A}_{(g_1,\ldots, g_k)}=\{f\in J\mid f\ \textrm{is a polynomial of type $(g_1,\ldots,g_k)$}\}$, where $\sum\limits_{i=1}^k {g_i}\leq n-1$ and $g_i\neq 0$.
\end{itemize}
Using Proposition \ref{Bgiutn} and Corollary \ref{Bg1gkfinito}, we get that there exist finite subsets $\mathcal{A}'_{g_i}\subseteq \mathcal{A}_{g_i}$, where $g_i\in\mathbb{Z}_n\setminus\{0\}$, and $\mathcal{A}'_{(g_1,\ldots,g_k)}\subseteq\mathcal{A}_{(g_1,\ldots, g_k)}$ such that
\begin{align*}
\langle\mathcal{A}_{g_i}\rangle_{T_{\mathbb{Z}_n}}&= \langle \mathcal{A}'_{g_i}\rangle_{T_{\mathbb{Z}_n}} \pmod{I}.\\
\langle\mathcal{A}_{(g_1,\ldots, g_k)}\rangle_{T_{\mathbb{Z}_n}} &= \langle \mathcal{A}'_{(g_1,\ldots, g_k)}\rangle_{T_{\mathbb{Z}_n}} \pmod{I}.
\end{align*}
It follows 
\[
J=\langle \mathcal{M} \rangle_{T_{\mathbb{Z}_n}} \pmod{I},
\]
where $\mathcal{M}$ is a finite set. Then $J=\langle \mathcal{M}\cup I \rangle_{T_{\mathbb{Z}_n}}$ and since $I$ has a finite basis, we can conclude that $J$ is finitely generated as a $T_{\mathbb{Z}_n}$-ideal.
\end{proof}

\section*{Acknowledgements}
We thank Alexei Krasilnikov for drawing our attention to his paper \cite{Kr0} and for pointing out several misprints in the text.


\begin{thebibliography}{9}

\bibitem{AB} E. Aljadeff, A. Kanel Belov, \textit{Representability and Specht problem for $G$-graded algebras}, Adv. Math. \textbf{225 (5)} (2010),  2391--2428.

\bibitem{baht_ol}
Yu. A. Bakhturin, A. Yu. Ol'shanski\u{\i}, \textit{Identical relations in finite Lie rings}, Mat. Sb., N. Ser. \textbf{96(138)} (1975) 543--559 (in Russian);  English Ttansl. Math. USSR, Sb. \textbf{25} (1975) 507--523.

\bibitem{CM} L. Centrone, F. Martino, \textit{A note on cocharacter sequence of Jordan upper triangular matrix algebra}, Commun. Algebra \textbf{45 (4)}  (2017), 1687--1695.

\bibitem{CMS} L. Centrone, F. Martino, M. Souza, \textit{Specht property for some varieties of Jordan algebras
of almost polynomial growth}, J. Algebra \textbf{(521)} (2019), 137--165.

\bibitem{dr_inf}
	V. Drensky, \textit{On identities in Lie algebras}, Algebra i Logika \textbf{13, No. 3} (1974), 265--290 (Russian); English transl. Algebra Logic \textbf{13, No. 3} (1974), 150--165.
	
\bibitem{fid_pk}
C. Fidelis, P. Koshlukov, \textit{$\mathbb{Z}$-graded identities of the Lie algebras $U_1$ in characteristic 2}, Math. Proc. Cambridge Phil. Soc., 2022, to appear. 
	
\bibitem{fkk}
J. A. Freitas, P. Koshlukov, A. Krasilnikov, \textit{$\mathbb{Z}$-graded identities of the Lie algebra $W_1$}, J. Algebra \textbf{427} (2015), 226--251. 

\bibitem{GS} A. Giambruno, M. S. Souza, \textit{Graded polynomial identities and Specht property of the Lie algebra $sl_2$}, J. Algebra. \textbf{(469)} (2017), 421--436.


\bibitem{Gr} A. V. Grishin, \textit{Examples of T-spaces and T-ideals in characteristic 2 without the finite basis property}, 
Fundam. Prikl. Mat. \textbf{5 (1)}  (1999),  101--118 (Russian).


\bibitem{Hig} G. Higman, \textit{Ordering by divisibility in abstract
algebras}, Proc. London Math. Soc. \textbf{3 (2)}  (1952), 326--336.

\bibitem{ilt_jordan}
A. V. Iltyakov, \textit{The Specht property of the ideals of identities of certain simple nonassociative algebras}, Algebra Logika \textbf{24, No. 3} (1985), 327--351 (Russian); English transl. Algebra Logic \textbf{24} (1985), 210--228 (1985).

\bibitem{ilt_alt}
A. V. Iltyakov, \textit{Finiteness of the basis of identities of a finitely generated alternative PI-algebra over a field of characteristic zero}, Sib. Mat. Zh. \textbf{32, No. 6(190)} (1991), 61--76 (Russian); English transl. Sib. Math. J. \textbf{32, No. 6} (1991), 948--961.

\bibitem{ilt}
A. V. Iltyakov, \textit{On finite basis of identities of Lie algebra representations}, Nova J. Algebra Geom. \textbf{1, No. 3} (1992), 207--259. 

\bibitem{Be} A. Kanel-Belov, \textit{Counterexamples to the Specht problem}, Sb. Math \textbf{191 (3--4)} (2000), 329--340.

\bibitem{Ke} A. R. Kemer, \textit{Varieties and $\mathbb{Z}_2$-graded algebras}, Izv. Ross. Akad. Nauk, Ser. Mat. \textbf{(48)} (1984),
1042--1059 (Russian); English transl. Math. USSR, Izv.  \textbf{(25)} (1985), 359--374.

\bibitem{pkja}
P. Koshlukov, \textit{Polynomial identities for a family of simple Jordan algebras}, Comm. Algebra  \textbf{16 (7)} (1988), 1325--1371. 

\bibitem{pk_sl2}
P. Koshlukov, \textit{Graded polynomial identities for the Lie algebra $sl_2(K)$}, Int. J. Algebra Comput. \textbf{18, No. 5} (2008), 825--836. 

\bibitem{KM} P. Koshlukov, F. Martino, \textit{Polynomial identities for the Jordan algebra of upper triangular matrices
of order two}, J. Pure Appl. Algebra \textbf{216 (11)}  (2012), 2524--2532.


\bibitem{KS} P. Koshlukov, D. Silva,  \textit{2-Graded polynomial identities for the Jordan algebra of the symmetric
matrices of order two}, J. Algebra \textbf{327 (1)}  (2011), 236--250.


\bibitem{KY} P. Koshlukov, F. Yukihide, \textit{Elementary gradings on the Lie algebra $UT_n^{(-)}$}, J. Algebra \textbf{(473)}  (2017), 66--79.

\bibitem{Kr0}
A. N. Krasilnikov, \textit{The finite basis property for certain varieties of Lie algebras}, Vestn. Mosk. Univ., Ser. I  \textbf{37, No. 2} (1982), 34--38 (Russian); English transl. Mosc. Univ. Math. Bull. \textbf{37, No. 2} (1982), 44--48.


\bibitem{Kr} A. N. Krasilnikov, \textit{Identities of Lie algebras with nilpotent commutator ideal over a field of finite characteristic}, Mat. Zametki \textbf{51} (1992), 47--52 (Russian); English transl. Math. Notes \textbf{3} (1992), 255-258.

\bibitem{kruse}
R. L. Kruse, \textit{Identities satisfied by a finite ring}, J. Algebra \textbf{26} (1973), 308--318. 

\bibitem{lvov}
I. V. Lvov, \textit{Varieties of associative rings}, Algebra i Logika \textbf{12 (3)} (1973), 269--297 (Russian); English transl.: Algebra Logic \textbf{12 (3)} (1973), 150--167.

\bibitem{MS} P. Morais, M. S. Souza, \textit{The algebra of $2\times 2$ upper triangular matrices as a commutative algebra: Gradings, graded polynomial and Specht property}, J. Algebra \textbf{593} (2022), 217--234.

\bibitem{oates_powell}
S. Oates, M. Powell, \textit{Identical relations in finite groups}, J. Algebra \textbf{1} (1964), 11--39.

\bibitem{razm1}
Yu. P. Razmyslov, \textit{Finite basing of the identities of a matrix algebra of second order over a field of characteristic zero}, Algebra Logika \textbf{12} (1973), 83--113 (Russian); English transl. Algebra Logic \textbf{12} (1973), 47--63 (1974).
 
\bibitem{razm2}
Yu. P. Razmyslov, \textit{ Existence of a finite base for certain varieties of algebras}, Algebra Logika \textbf{13, No. 6} (1974), 685--693 (Russian); English transl. Algebra Logic \textbf{13} (1974), 394--399 (1975).

\bibitem{VS} V. V. Shchigolev, \textit{Examples of infinitely based T-spaces}, Mat. Sb. \textbf{191 (3)}  (2000) 143--160 (Russian); English transl.  Sb. Math. \textbf{191 (3)}  (2000).


\bibitem{SS} D. Silva,  M. S. Souza, \textit{Specht property for the $2$-graded identities of the Jordan algebra of a bilinear form}, Comm. Algebra \textbf{45 (4)}  (2017), 1618--1626.

\bibitem{IS} I. Sviridova, \textit{Identities of PI-algebras graded by a finite Abelian group}, Commun. Algebra \textbf{39 (9)} (2011), 3462--3490.

\bibitem{VL} M. R. Vaughan-Lee, \textit{Varieties of Lie algebras}, Q. J. Math. Oxf. Ser. \textbf{(2) 21}  (1970), 297--308.


\bibitem{vais_zelm}
A. Ya. Va\u{\i}s, E. Zel'manov, \textit{Kemer's theorem for finitely generated Jordan algebras}, Izv. Vyssh. Uchebn. Zaved. Mat. \textbf{No. 6(325)} (1989), 42--51; English transl. Sov. Math. \textbf{33, No. 6} (1990), 38--47.

\bibitem{vas}
S. Yu. Vasilovskij, \textit{Basis of identities of the Jordan algebra of a bilinear form over an infinite field}, Tr. Inst. Mat. \textbf{16} (1989), 5--37 (Russian); English transl. Sib. Adv. Math. \textbf{1, No. 4} (1991), 142--185. 

\end{thebibliography}
\end{document}